\documentclass[11pt,a4paper]{article}
\usepackage{amsmath, amsfonts, amsthm, amssymb, graphicx, xcolor, dsfont, enumerate, dcolumn,subcaption}
\usepackage[utf8]{inputenc}
\usepackage[numbers,square, comma, sort&compress]{natbib}
\usepackage{thmtools} 
\usepackage{thm-restate}
\usepackage{threeparttable}
\usepackage[T1]{fontenc}
\usepackage[ruled,vlined]{algorithm2e}
\usepackage[english]{babel}
\usepackage[a4paper,top=2cm,bottom=2cm,left=3cm,right=3cm,marginparwidth=1.75cm]{geometry}
\usepackage{amsmath,amsthm}
\usepackage[colorlinks=true, allcolors=blue]{hyperref}
\usepackage{xspace}
\usepackage{tikz}
\usepackage{nicefrac}

\newtheoremstyle{mystyle}
  {}
  {}
  {\itshape}
  {}
  {\bfseries}
  {.}
  { }
  {}
\theoremstyle{mystyle}

\newtheorem{theorem}{Theorem}[section]
\newtheorem{lemma}[theorem]{Lemma}
\newtheorem{corollary}[theorem]{Corollary}
\newtheorem{observation}[theorem]{Observation}
\newtheorem{conjecture}[theorem]{Conjecture}
\newtheorem{claim}{Claim}[theorem]

\newenvironment{claimproof}{\noindent {\emph{Proof of Claim.}}}{\hfill$\blacksquare$\medskip}

\newcommand{\cA}{\mathcal{A}}
\newcommand{\cB}{\mathcal{B}}
\newcommand{\colo}[1]{\ensuremath{#1}-\textsc{Coloring}\xspace}
\newcommand{\NP}{\textsf{NP}}
\newcommand{\obs}[1]{\mathsf{Obstructions}(#1)}
\newcommand{\hull}{\mathsf{hull}}

\begin{document}

\title{Minimal obstructions to $C_5$-coloring in hereditary graph classes\footnote{An extended abstract of this paper appeared in the proceedings of the 49th International Symposium on Mathematical Foundations of Computer Science (MFCS 2024)~\cite{MFCS2024}.}}
\author{{\sc Jan GOEDGEBEUR\footnote{Department of Computer Science, KU Leuven Campus Kulak-Kortrijk, 8500 Kortrijk, Belgium}\;\footnote{Department of Applied Mathematics, Computer Science and Statistics, Ghent University, 9000 Ghent, Belgium}\;,
Jorik JOOKEN\footnotemark[1]\;,
Karolina OKRASA\footnote{Warsaw University of Technology, Poland}\;, }\\[1mm]
Pawe{\l} \MakeUppercase{Rz{\k{a}}{\.z}ewski}\footnote{Warsaw University of Technology and University of Warsaw, Poland. Pawe{\l} Rz{\k{a}}{\.z}ewski is supported by a Polish National Science Centre grant no. 2018/31/D/ST6/00062.}\;\,,
and {Oliver SCHAUDT\footnote{University of Cologne, Germany}}\;\footnote{E-mail addresses: jan.goedgebeur@kuleuven.be; jorik.jooken@kuleuven.be; karolina.okrasa@pw.edu.pl; pawel.rzazewski@pw.edu.pl;
oliver.schaudt@bayer.com}}
\date{}

\maketitle

\begin{abstract}
For graphs $G$ and $H$, an $H$-coloring of $G$ is an edge-preserving mapping from $V(G)$ to $V(H)$.
Note that if $H$ is the triangle, then $H$-colorings are equivalent to $3$-colorings.
In this paper we are interested in the case that $H$ is the five-vertex cycle $C_5$.

A minimal obstruction to $C_5$-coloring is a graph that does not have a $C_5$-coloring, but every proper induced subgraph thereof has a $C_5$-coloring. In this paper we are interested in minimal obstructions to $C_5$-coloring in $F$-free graphs, i.e., graphs that exclude some fixed graph $F$ as an induced subgraph. Let $P_t$ denote the path on $t$ vertices, and let $S_{a,b,c}$ denote the graph obtained from paths $P_{a+1},P_{b+1},P_{c+1}$ by identifying one of their endvertices.

We show that there is only a finite number of minimal obstructions to $C_5$-coloring among $F$-free graphs, where $F \in \{ P_8, S_{2,2,1}, S_{3,1,1}\}$ and explicitly determine all such obstructions. This extends the results of Kami\'{n}ski and Pstrucha [Discr.\ Appl.\ Math.\ 261, 2019] who proved that there is only a finite number of $P_7$-free minimal obstructions to $C_5$-coloring, and of D\k{e}bski et al.\ [ISAAC 2022 Proc.] who showed that the triangle is the unique $S_{2,1,1}$-free minimal obstruction to $C_5$-coloring.

We complement our results with a construction of an infinite family of minimal obstructions to $C_5$-coloring, which are simultaneously $P_{13}$-free and $S_{2,2,2}$-free.
We also discuss infinite families of $F$-free minimal obstructions to $H$-coloring for other graphs $H$.

\bigskip\noindent \textbf{Keywords:} graph homomorphism, critical graphs, hereditary graph classes

\smallskip\noindent \textbf{MSC 2020:} 05C15, 05C85, 68W05, 68R10
\end{abstract}


\section{Introduction}
Out of a great number of interesting and elegant graph problems, the notion of graph coloring is, arguably, among the most popular and well-studied ones, not only from combinatorial, but also from algorithmic point of view.
For an integer $k \geq 1$, a \emph{$k$-coloring} of a graph $G$ is a function $c: V(G) \to \{1,\ldots,k\}$ such that for every edge $uv \in E(G)$ it holds that $c(u)\neq c(v)$.
For a fixed integer $k \geq 1$, the \colo{k} problem is a computational problem in which an instance is a graph $G$ and we ask  whether there exists a $k$-coloring of $G$.

The \colo{k} problem is known to be polynomial-time solvable if $k \leq 2$ and  \NP-complete for $k \geq 3$. 
Still, even if $k \geq 3$, it is often possible to obtain polynomial-time algorithms that solve \colo{k} if we somehow restrict the class of input instances, for example, to perfect graphs~\cite{DBLP:journals/combinatorica/GrotschelLS81,DBLP:journals/combinatorica/GrotschelLS84}, bounded-treewidth graphs~\cite{DBLP:journals/dam/ArnborgP89} or intersection graphs of geometric objects~\cite{DBLP:journals/siammax/GareyJMP80}.

Observe that these example classes are \emph{hereditary}, i.e., closed under deleting vertices.
Such a property is very useful in algorithm design, as it combines well with standard algorithmic techniques, like branching or divide-\&-conquer.
Therefore, if we want to study some computational problem in a restricted graph class $\mathcal{G}$,
choosing $\mathcal{G}$ to be hereditary appears to be reasonable.
For a fixed graph ${F}$, we say that a graph $G$ is \emph{$F$-free} if it does not contain $F$ as an induced subgraph.
If $\mathcal{F}$ is a family of graphs, we say that a graph $G$ is \emph{$\mathcal{F}$-free} if $G$ is $F$-free for every $F \in \mathcal{F}$.
Each hereditary class of graphs can be characterized by a (possibly infinite) family $\mathcal{F}$ of forbidden induced subgraphs. 

\label{sect:intro}
\paragraph{Coloring hereditary graph classes.}  
As a first step towards understanding the complexity of \colo{k} in hereditary classes it is natural to consider classes defined by a single induced subgraph $F$.
If $F$ contains a cycle or a vertex of degree at least 3, it follows from the classical results by Emden-Weinert~\cite{DBLP:journals/cpc/Emden-WeinertHK98}, Holyer~\cite{Holyer1981TheNO}, and Leven and Galil~\cite{DBLP:journals/jal/LevenG83} that for every $k \geq 3$, \colo{k} remains \NP-complete when restricted to $F$-free graphs.
Thus we are left with the case that $F$ is a \emph{linear forest}, i.e., every component of $F$ is a path.

However, if $F$ is a linear forest, the situation becomes more complicated.
For simplicity, let us focus on the case when $F$ is connected, i.e., $F$ is a path on $t$ vertices, denoted by $P_t$.
Then, \colo{k} is polynomial-time-solvable in $P_t$-free graphs if $t \leq 5$, or if $(k,t) \in \{(3,6),(3,7),(4,6)\}$~\cite{DBLP:journals/ejc/Huang16,DBLP:conf/soda/SpirklCZ19,DBLP:journals/combinatorica/BonomoCMSSZ18,DBLP:journals/algorithmica/HoangKLSS10}. 
On the other hand, for any $k \geq 4$, the \colo{k} problem is \NP-complete in $P_t$-free graphs for all other values of $t$~\cite{DBLP:journals/ejc/Huang16}.
The complexity of the remaining cases, i.e., \colo{3} of $P_t$-free graphs where $t \geq 8$, remains unknown:
we do not know polynomial-time algorithms nor any hardness proofs.
The general belief is that all these cases are in fact tractable, which is supported by the existence of a quasipolynomial-time algorithm for \colo{3} in $P_t$-free graphs, for every fixed $t$~\cite{DBLP:conf/sosa/PilipczukPR21}. For the summary of the results on the complexity of \colo{k} $P_t$-free graphs see Figure~\ref{fig:coloring}. 
Let us remark that there are also some results for disconnected forbidden linear forests~\cite{DBLP:journals/algorithmica/KlimosovaMMNPS20,DBLP:journals/algorithmica/ChudnovskyHSZ21}.

\paragraph{Minimal obstructions to $k$-coloring.} 
One can look at \colo{k} from another, purely combinatorial perspective.
Instead of asking whether a graph $G$ admits a $k$-coloring, we can ask whether it contains a \emph{dual object}, i.e., some structure that forces the chromatic number to be at least $k+1$.
For example, \colo{2} can be equivalently expressed as a question whether a graph contains an odd cycle.
As another example, \colo{k} restricted to perfect graphs is equivalent to the question of the existence of a $(k+1)$-clique, i.e., the complete graph on $k+1$ vertices, denoted by $K_{k+1}$.

In other words, odd cycles are \emph{minimal non-2-colorable graphs} and $(k+1)$-cliques are \emph{minimal non-$k$-colorable perfect graphs} (where minimality is defined with respect to the induced subgraph relation).
Formally, if a graph $G$ is not $k$-colorable, but every proper induced subgraph of it is $k$-colorable, we say that $G$ is \emph{vertex-$(k+1)$-critical} or is a \emph{minimal obstruction to $k$-coloring}.
We denote by $\obs{k}$ the set of all minimal obstructions to $k$-coloring.
Note that $\obs{k}$ naturally forms a family of dual objects---a graph is $k$-colorable if and only if it does not contain any graph from $\obs{k}$ as an induced subgraph.

Suppose that, for some $k$, there is a polynomial-time algorithm $\mathsf{Alg}_k$ that takes as an input a graph $G$ and answers whether it contains any graph from $\obs{k}$ as an induced subgraph (i.e., whether $G$ is \emph{not} $\obs{k}$-free).
From the discussion above it follows that the existence of $\mathsf{Alg}_k$ yields a polynomial-time algorithm for \colo{k}.
Thus it is unlikely that $\mathsf{Alg}_k$ exists for any $k \geq 3$.
However, it is still possible when we restrict the input graphs to a certain class $\mathcal{G}$ (like perfect graphs in the example above).
Recall that we are interested in the case that $\mathcal{G}=F\text{-free}$, where $F$ is a path.
Let us denote such a restriction $\mathsf{Alg}_k$ to $F$-free graphs by $\mathsf{Alg}_{k,F}$. 

Note that the existence of $\mathsf{Alg}_{k,F}$ is trivial if $(\obs{k} \cap F\text{-free})$ is finite; indeed, brute force works in this case. This line of arguments allows us to further refine cases that are polynomial-time solvable:
into pairs $(k,F)$, where $(\obs{k} \cap F\text{-free})$ is finite, and the others.
Recall that the algorithm for \colo{k} obtained for the former ones is able to produce a \emph{negative certificate}: a small (constant-size) witness that the input graph is \emph{not} $k$-colorable. We refer the reader to the survey of McConnell et al.~\cite{DBLP:journals/csr/McConnellMNS11} for more information about certifying algorithms.

It turns out that we can fully characterize all pairs $(k,F)$ for which $\obs{k} \cap F\text{-free}$ is finite.
It is well-known that $P_4$-free graphs (also known as \emph{cographs}) are perfect and thus the only minimal obstruction to $k$-coloring is the $(k+1)$-clique.
Bruce et al.~\cite{DBLP:conf/isaac/BruceHS09} proved that there is a finite number of minimal obstructions to 3-coloring among $P_5$-free graphs.
The result was later extended by Chudnovsky et al.~\cite{DBLP:journals/jct/ChudnovskyGSZ20} who showed that the family of $P_6$-free minimal obstructions to 3-coloring is  is also finite, and that this is no longer true among $P_7$-free graphs (and thus for $P_t$-free graphs for every $t \geq 7$).
If the number of colors is larger, things get more difficult faster: Hoàng et al.~\cite{DBLP:journals/dam/HoangMRSV15} showed that for each $k\geq 4$ there exists an infinite family of $P_5$-free minimal obstructions to $k$-coloring. See also Figure~\ref{fig:coloring}.

\begin{figure}
    \centering
    \includegraphics[page=6,scale=0.8]{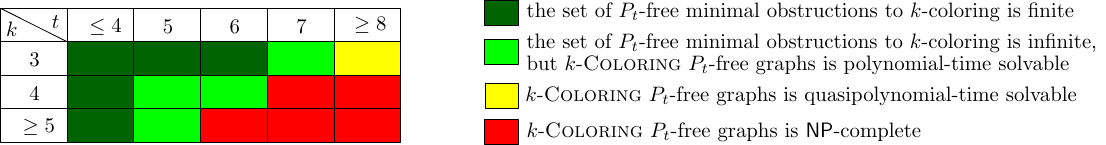}
    \caption{The complexity of \colo{k} $P_t$-free graphs.}
    \label{fig:coloring}
\end{figure}

\paragraph{$H$-coloring $F$-free graphs and minimal obstructions to $H$-colorings.} Graph colorings can be seen as a special case of \emph{graph homomorphisms}.
For graphs $G$ and $H$, an \emph{$H$-coloring} of $G$ is a function $c: V(G) \to V(H)$ such that for every edge $uv \in E(G)$ it holds that $c(u)c(v) \in E(H)$. The graph $H$ is usually called the \emph{target graph}.
It is straightforward to verify that homomorphisms from $G$ to the $k$-clique, are in one-to-one correspondence to $k$-colorings of $G$. For this reason one often refers to the vertices of $H$ as \emph{colors}.

For a fixed graph $H$, by \colo{H} we denote the computational problem that takes as an input a graph $G$ and asks whether $G$ admits a homomorphism to $H$. The complexity dichotomy for \colo{H} was proven by Hell and Ne\v{s}et\v{r}il~\cite{DBLP:journals/jct/HellN90}: the problem is polynomial-time solvable if $H$ is bipartite, and \NP-complete otherwise.

The complexity landscape of \colo{H} in $F$-free graphs for non-complete target graphs is far from being fully understood.
Chudnovsky et al.~\cite{DBLP:journals/iandc/ChudnovskyHRSZ23} proved that if $H$ is an odd cycle on at least 5 vertices, then \colo{H} is polynomial-time solvable in $P_9$-free graphs; they also showed a number of hardness results for more general variants of the homomorphism problem.
Feder and Hell~\cite{edgelists} and D\k{e}bski et al.~\cite{debski_et_al:LIPIcs.ISAAC.2022.14} studied the case when $H$ is an \emph{odd wheel}, i.e., an odd cycle with universal vertex added.
The most general algorithmic results were provided by Okrasa and Rzążewski~\cite{DBLP:conf/stacs/OkrasaR21} who showed that 
\begin{enumerate}[(OR1)]
    \item if $H$ does not contain $C_4$ as a subgraph, then \colo{H} can be solved in quasipolynomial time in $P_t$-free graphs for any fixed $t$ (note that a better running time here would also mean progress for  \colo{3} $P_t$-free graphs),
    \item if $H$ is of girth at least 5, then \colo{H} can be solved in subexponential time in $F$-free graphs, where $F$ is any fixed \emph{subdivided claw}, i.e., any graph obtained from the three-leaf star by subdividing edges.
\end{enumerate}

While these are not polynomial-time algorithms, no \NP-hardness proofs for these cases are known either.
To complete the picture, from~\cite{DBLP:conf/stacs/OkrasaR21} it also follows that if $H$ is a so-called \emph{projective core} that  contains $C_4$ as a subgraph, then there exists a $t$ such that \colo{H} is \NP-complete in $P_t$-free graphs (and thus also in graphs excluding some fixed subdivided claw). Furthermore, the hardness reductions even exclude any subexponential-time algorithms for these cases, assuming the Exponential Time Hypothesis (ETH).
Let us skip the definition of a projective core, as it is quite technical and not really relevant for this paper. However, it is worth pointing out that almost all graphs are projective cores~\cite{DBLP:journals/jgt/LuczakN04,DBLP:journals/dm/HellN92}.

Since we are interested in a finer classification of polynomial-time-solvable cases, we should be looking at pairs $(H,F)$ of graphs such that the \colo{H} problem is \emph{not} known to be \NP-complete in $F$-free graphs. From the discussion above it follows that there are two natural families of such pairs to consider: 
\begin{enumerate}[(i)]
    \item when $H$ does not contain $C_4$ as a subgraph and $F$ is a path,
    \item when $H$ is of girth at least 5 and $F$ is a subdivided claw.
\end{enumerate}

It is straightforward to generalize the notion of minimal obstructions to the setting of $H$-colorings.
A graph $G$ is called a \emph{minimal obstruction to $H$-coloring} if there is no $H$-coloring of $G$, but every proper induced subgraph of $G$ can be $H$-colored.

The area of minimal obstructions to $H$-coloring is rather unexplored.
In the setting of (i), Kami\'{n}ski and Pstrucha~\cite{DBLP:journals/dam/KaminskiP19} showed that for any $t\geq 5$, there are  finitely many minimal obstructions to $C_{t-2}$-coloring among $P_t$-free graphs.\footnote{While in~\cite{DBLP:journals/dam/KaminskiP19} the authors consider minimality with respect to the \emph{subgraph} relation, it is not hard to observe that bounded number of subgraph-minimal obstructions is equivalent to the bounded number of induced-subgraph-minimal obstructions.}
In particular, the family of $P_7$-free minimal obstructions to $C_5$-coloring is finite. 
In the setting of (ii), D\k{e}bski et al.~\cite{debski_et_al:LIPIcs.ISAAC.2022.14} showed that the triangle is the only minimal obstruction to $C_5$-coloring among graphs that exclude the \emph{fork}, i.e., the graph obtained from the three-leaf star by subdividing one edge once.

\paragraph{Our contribution.}
As our first result, we show the following strengthening of the result of Kami\'{n}ski and Pstrucha.
  \begin{restatable}{theorem}{thmfinitepaths}
  \label{thm:finite-paths}
    There are  19 minimal obstructions to $C_5$-coloring among $P_8$-free graphs.
  \end{restatable}

Let us sketch the proof of Theorem~\ref{thm:finite-paths}. Note that $K_3$ is a minimal obstruction for $C_5$-coloring, so from now on we focus on graphs that are $\{P_8,K_3\}$-free.
For $i \in \{1,2,3,4\}$, by $Q_i$ we denote the graph obtained from two copies of $C_5$ by identifying an $i$-vertex subpath of one cycle with an $i$-vertex subpath of the other one, see Figure~\ref{fig:qs}.
In the proof we separately consider minimal obstructions that contain some $Q_i$ as an induced subgraph, and those that are $\{Q_1,Q_2,Q_3,Q_4\}$-free.

\begin{figure}
    \centering
    \includegraphics[page=4,scale=0.7]{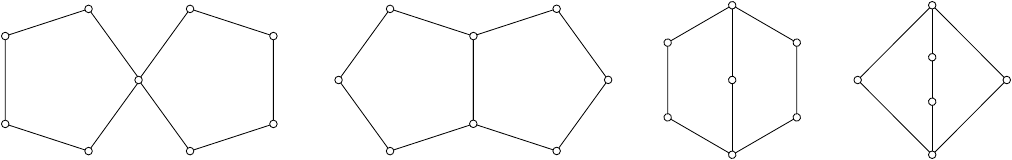}
    \caption{Graphs $Q_1, Q_2, Q_3$, and $Q_4$ (left to right).}
    \label{fig:qs}
\end{figure}

The intuition behind this is as follows. Notice that if $G$ contains an induced 5-vertex cycle, the vertices of this cycle must be mapped to the vertices of $C_5$ bijectively, respecting the ordering along the cycle.
Consequently, if $G$ contains some $Q_i$, the colorings of the vertices in $Q_i$ are somehow restricted. Combining several $Q_i$'s we might impose some contradictory constraints and thus build a graph that is not $C_5$-colorable.
However, as each $Q_i$ already contains quite long induced paths, we might hope that by combining several $Q_i$ we are either forced to create an induced $P_8$ (if we add only few edges between different $Q_i$'s) or $K_3$ (if we add too many such edges).
Thus the possibilities of building non-$C_5$-colorable graphs using this approach are somehow limited.
It turns out that this intuition is correct: there are 18 graphs that are $\{P_8,K_3\}$-free and contain $Q_i$, for some $i \in \{1,2,3,4\}$, as an induced subgraph. Together with $K_3$, they are shown in Figure~\ref{fig:p8FreeObstructions}.
This part of the proof is computer-aided.

\begin{figure}[h!]
\begin{center}
\begin{tikzpicture}[scale=0.28]
  \def\sides{3}
  \def\radius{3}

  \foreach \i in {1,...,\sides} {
    \fill ({360/\sides * \i}:\radius) circle (4pt);
  }
  
  \foreach \i in {1,2,3} {
    \draw ({360/\sides * (\i + 1)}:\radius) -- ({360/\sides * \i}:\radius);
  }
\end{tikzpicture} \quad
\begin{tikzpicture}[scale=0.28]
  \def\sides{8}
  \def\radius{3}

  \foreach \i in {1,...,\sides} {
    \fill ({360/\sides * \i}:\radius) circle (4pt);
  }
  
  \foreach \i in {1,3,4,5,7,8} {
    \draw ({360/\sides * (\i + 1)}:\radius) -- ({360/\sides * \i}:\radius);
    \draw ({360/\sides * (\i + 4)}:\radius) -- ({360/\sides * \i}:\radius);
    \draw ({360/\sides * (\i + 7)}:\radius) -- ({360/\sides * \i}:\radius);
  }
  \foreach \i in {2,6} {
    \draw ({360/\sides * (\i + 1)}:\radius) -- ({360/\sides * \i}:\radius);
    \draw ({360/\sides * (\i + 7)}:\radius) -- ({360/\sides * \i}:\radius);
  }
\end{tikzpicture}
\quad
\begin{tikzpicture}[scale=0.28]
  \def\sides{8}
  \def\radius{3}

  \foreach \i in {1,...,\sides} {
    \fill ({360/\sides * \i}:\radius) circle (4pt);
  }
  
  \foreach \i in {1,2,3,4,5,6,7,8} {
    \draw ({360/\sides * (\i + 1)}:\radius) -- ({360/\sides * \i}:\radius);
    \draw ({360/\sides * (\i + 4)}:\radius) -- ({360/\sides * \i}:\radius);
    \draw ({360/\sides * (\i + 7)}:\radius) -- ({360/\sides * \i}:\radius);
  }
\end{tikzpicture}
\quad
\begin{tikzpicture}[scale=0.28]
  \def\sides{8}
  \def\radius{3}

  \foreach \i in {1,...,\sides} {
    \fill ({360/\sides * \i}:\radius) circle (4pt);
  }
  
  \foreach \i in {1,3,5,7} {
    \draw ({360/\sides * (\i + 1)}:\radius) -- ({360/\sides * \i}:\radius);
    \draw ({360/\sides * (\i + 4)}:\radius) -- ({360/\sides * \i}:\radius);
    \draw ({360/\sides * (\i + 7)}:\radius) -- ({360/\sides * \i}:\radius);
  }
  \foreach \i in {2,4,6,8} {
    \draw ({360/\sides * (\i + 1)}:\radius) -- ({360/\sides * \i}:\radius);
    \draw ({360/\sides * (\i + 7)}:\radius) -- ({360/\sides * \i}:\radius);
  }
\end{tikzpicture}\\
\begin{tikzpicture}[scale=0.28]
  \def\sides{8}
  \def\radius{3}

  \foreach \i in {1,...,\sides} {
    \fill ({360/\sides * \i}:\radius) circle (4pt);
  }
  
  \foreach \i in {3,4,7,8} {
    \draw ({360/\sides * (\i + 1)}:\radius) -- ({360/\sides * \i}:\radius);
    \draw ({360/\sides * (\i + 4)}:\radius) -- ({360/\sides * \i}:\radius);
    \draw ({360/\sides * (\i + 7)}:\radius) -- ({360/\sides * \i}:\radius);
  }
  \foreach \i in {1,2,5,6} {
    \draw ({360/\sides * (\i + 1)}:\radius) -- ({360/\sides * \i}:\radius);
    \draw ({360/\sides * (\i + 7)}:\radius) -- ({360/\sides * \i}:\radius);
  }
\end{tikzpicture} \quad
\begin{tikzpicture}[scale=0.28]
  \def\sides{13}
  \def\radius{3}

  \foreach \i in {1,...,\sides} {
    \fill ({360/\sides * \i}:\radius) circle (4pt);
  }
  
  \foreach \i in {1,2,3,4,5,6,7,8,9,10,11,12,13} {
    \draw ({360/\sides * (\i + 1)}:\radius) -- ({360/\sides * \i}:\radius);
    \draw ({360/\sides * (\i + 4)}:\radius) -- ({360/\sides * \i}:\radius);
    \draw ({360/\sides * (\i + 9)}:\radius) -- ({360/\sides * \i}:\radius);
    \draw ({360/\sides * (\i + 12)}:\radius) -- ({360/\sides * \i}:\radius);
  }
\end{tikzpicture}\\
\begin{tikzpicture}[scale=0.28]
  \def\sides{9}
  \def\radius{3}

  \foreach \i in {1,...,\sides} {
    \fill ({360/\sides * \i}:\radius) circle (4pt);
  }
  
  \foreach \i in {1} {
    \draw ({360/\sides * (\i + 1)}:\radius) -- ({360/\sides * \i}:\radius);
    \draw ({360/\sides * (\i + 7)}:\radius) -- ({360/\sides * \i}:\radius);
  }
  \foreach \i in {2} {
    \draw ({360/\sides * (\i + 1)}:\radius) -- ({360/\sides * \i}:\radius);
        \draw ({360/\sides * (\i + 4)}:\radius) -- ({360/\sides * \i}:\radius);
    \draw ({360/\sides * (\i + 8)}:\radius) -- ({360/\sides * \i}:\radius);
  }
    \foreach \i in {3} {
    \draw ({360/\sides * (\i + 1)}:\radius) -- ({360/\sides * \i}:\radius);
        \draw ({360/\sides * (\i + 6)}:\radius) -- ({360/\sides * \i}:\radius);
    \draw ({360/\sides * (\i + 8)}:\radius) -- ({360/\sides * \i}:\radius);
  }
    \foreach \i in {4} {
    \draw ({360/\sides * (\i + 1)}:\radius) -- ({360/\sides * \i}:\radius);
    \draw ({360/\sides * (\i + 8)}:\radius) -- ({360/\sides * \i}:\radius);
  }
      \foreach \i in {5} {
    \draw ({360/\sides * (\i + 1)}:\radius) -- ({360/\sides * \i}:\radius);
    \draw ({360/\sides * (\i + 8)}:\radius) -- ({360/\sides * \i}:\radius);
  }
      \foreach \i in {6} {
    \draw ({360/\sides * (\i + 1)}:\radius) -- ({360/\sides * \i}:\radius);
        \draw ({360/\sides * (\i + 5)}:\radius) -- ({360/\sides * \i}:\radius);
    \draw ({360/\sides * (\i + 8)}:\radius) -- ({360/\sides * \i}:\radius);
  }
        \foreach \i in {7} {
    \draw ({360/\sides * (\i + 1)}:\radius) -- ({360/\sides * \i}:\radius);
    \draw ({360/\sides * (\i + 8)}:\radius) -- ({360/\sides * \i}:\radius);
  }
        \foreach \i in {8} {
    \draw ({360/\sides * (\i + 1)}:\radius) -- ({360/\sides * \i}:\radius);
        \draw ({360/\sides * (\i + 2)}:\radius) -- ({360/\sides * \i}:\radius);
    \draw ({360/\sides * (\i + 8)}:\radius) -- ({360/\sides * \i}:\radius);
  }
        \foreach \i in {9} {
    \draw ({360/\sides * (\i + 3)}:\radius) -- ({360/\sides * \i}:\radius);
    \draw ({360/\sides * (\i + 8)}:\radius) -- ({360/\sides * \i}:\radius);
  }
\end{tikzpicture} \quad
\begin{tikzpicture}[scale=0.28]
  \def\sides{10}
  \def\radius{3}

  \foreach \i in {1,...,\sides} {
    \fill ({360/\sides * \i}:\radius) circle (4pt);
  }
  
  \foreach \i in {1} {
    \draw ({360/\sides * (\i + 1)}:\radius) -- ({360/\sides * \i}:\radius);
    \draw ({360/\sides * (\i + 9)}:\radius) -- ({360/\sides * \i}:\radius);
  }
  \foreach \i in {2} {
    \draw ({360/\sides * (\i + 1)}:\radius) -- ({360/\sides * \i}:\radius);
        \draw ({360/\sides * (\i + 6)}:\radius) -- ({360/\sides * \i}:\radius);
    \draw ({360/\sides * (\i + 9)}:\radius) -- ({360/\sides * \i}:\radius);
  }
    \foreach \i in {3} {
    \draw ({360/\sides * (\i + 1)}:\radius) -- ({360/\sides * \i}:\radius);
    \draw ({360/\sides * (\i + 9)}:\radius) -- ({360/\sides * \i}:\radius);
  }
    \foreach \i in {4} {
    \draw ({360/\sides * (\i + 1)}:\radius) -- ({360/\sides * \i}:\radius);
    \draw ({360/\sides * (\i + 9)}:\radius) -- ({360/\sides * \i}:\radius);
  }
      \foreach \i in {5} {
    \draw ({360/\sides * (\i + 1)}:\radius) -- ({360/\sides * \i}:\radius);
    \draw ({360/\sides * (\i + 4)}:\radius) -- ({360/\sides * \i}:\radius);
        \draw ({360/\sides * (\i + 9)}:\radius) -- ({360/\sides * \i}:\radius);
  }
      \foreach \i in {6} {
    \draw ({360/\sides * (\i + 1)}:\radius) -- ({360/\sides * \i}:\radius);
        \draw ({360/\sides * (\i + 9)}:\radius) -- ({360/\sides * \i}:\radius);
  }
        \foreach \i in {7} {
    \draw ({360/\sides * (\i + 1)}:\radius) -- ({360/\sides * \i}:\radius);
    \draw ({360/\sides * (\i + 3)}:\radius) -- ({360/\sides * \i}:\radius);
        \draw ({360/\sides * (\i + 9)}:\radius) -- ({360/\sides * \i}:\radius);
  }
        \foreach \i in {8} {
        \draw ({360/\sides * (\i + 4)}:\radius) -- ({360/\sides * \i}:\radius);
    \draw ({360/\sides * (\i + 9)}:\radius) -- ({360/\sides * \i}:\radius);
  }
        \foreach \i in {9} {
    \draw ({360/\sides * (\i + 1)}:\radius) -- ({360/\sides * \i}:\radius);
    \draw ({360/\sides * (\i + 6)}:\radius) -- ({360/\sides * \i}:\radius);
  }
          \foreach \i in {10} {
    \draw ({360/\sides * (\i + 1)}:\radius) -- ({360/\sides * \i}:\radius);
    \draw ({360/\sides * (\i + 7)}:\radius) -- ({360/\sides * \i}:\radius);
        \draw ({360/\sides * (\i + 9)}:\radius) -- ({360/\sides * \i}:\radius);
  }
\end{tikzpicture} \quad
\begin{tikzpicture}[scale=0.28]
  \def\sides{13}
  \def\radius{3}

  \foreach \i in {1,...,\sides} {
    \fill ({360/\sides * \i}:\radius) circle (4pt);
  }
  
  \foreach \i in {1,7,13} {
    \draw ({360/\sides * (\i + 1)}:\radius) -- ({360/\sides * \i}:\radius);
    \draw ({360/\sides * (\i + 9)}:\radius) -- ({360/\sides * \i}:\radius);
    \draw ({360/\sides * (\i + 12)}:\radius) -- ({360/\sides * \i}:\radius);
  }
    \foreach \i in {2,4,8,11} {
    \draw ({360/\sides * (\i + 1)}:\radius) -- ({360/\sides * \i}:\radius);
    \draw ({360/\sides * (\i + 12)}:\radius) -- ({360/\sides * \i}:\radius);
  }
      \foreach \i in {3,9,10} {
    \draw ({360/\sides * (\i + 1)}:\radius) -- ({360/\sides * \i}:\radius);
        \draw ({360/\sides * (\i + 4)}:\radius) -- ({360/\sides * \i}:\radius);
            \draw ({360/\sides * (\i + 9)}:\radius) -- ({360/\sides * \i}:\radius);
    \draw ({360/\sides * (\i + 12)}:\radius) -- ({360/\sides * \i}:\radius);
  }
        \foreach \i in {5,6,12} {
    \draw ({360/\sides * (\i + 1)}:\radius) -- ({360/\sides * \i}:\radius);
        \draw ({360/\sides * (\i + 4)}:\radius) -- ({360/\sides * \i}:\radius);
            \draw ({360/\sides * (\i + 12)}:\radius) -- ({360/\sides * \i}:\radius);
  }
\end{tikzpicture} \quad
\begin{tikzpicture}[scale=0.28]
  \def\sides{13}
  \def\radius{3}

  \foreach \i in {1,...,\sides} {
    \fill ({360/\sides * \i}:\radius) circle (4pt);
  }
  
  \foreach \i in {1,7,8} {
    \draw ({360/\sides * (\i + 1)}:\radius) -- ({360/\sides * \i}:\radius);
    \draw ({360/\sides * (\i + 9)}:\radius) -- ({360/\sides * \i}:\radius);
    \draw ({360/\sides * (\i + 12)}:\radius) -- ({360/\sides * \i}:\radius);
  }
    \foreach \i in {2,11} {
    \draw ({360/\sides * (\i + 1)}:\radius) -- ({360/\sides * \i}:\radius);
    \draw ({360/\sides * (\i + 12)}:\radius) -- ({360/\sides * \i}:\radius);
  }
      \foreach \i in {3,4,9,10,13} {
    \draw ({360/\sides * (\i + 1)}:\radius) -- ({360/\sides * \i}:\radius);
        \draw ({360/\sides * (\i + 4)}:\radius) -- ({360/\sides * \i}:\radius);
            \draw ({360/\sides * (\i + 9)}:\radius) -- ({360/\sides * \i}:\radius);
    \draw ({360/\sides * (\i + 12)}:\radius) -- ({360/\sides * \i}:\radius);
  }
        \foreach \i in {5,6,12} {
    \draw ({360/\sides * (\i + 1)}:\radius) -- ({360/\sides * \i}:\radius);
        \draw ({360/\sides * (\i + 4)}:\radius) -- ({360/\sides * \i}:\radius);
            \draw ({360/\sides * (\i + 12)}:\radius) -- ({360/\sides * \i}:\radius);
  }
\end{tikzpicture} \quad
\begin{tikzpicture}[scale=0.28]
  \def\sides{13}
  \def\radius{3}

  \foreach \i in {1,...,\sides} {
    \fill ({360/\sides * \i}:\radius) circle (4pt);
  }
  
  \foreach \i in {1,3,6,7,10,12,13} {
    \draw ({360/\sides * (\i + 1)}:\radius) -- ({360/\sides * \i}:\radius);
        \draw ({360/\sides * (\i + 4)}:\radius) -- ({360/\sides * \i}:\radius);
    \draw ({360/\sides * (\i + 9)}:\radius) -- ({360/\sides * \i}:\radius);
    \draw ({360/\sides * (\i + 12)}:\radius) -- ({360/\sides * \i}:\radius);
  }
    \foreach \i in {2,8,9} {
    \draw ({360/\sides * (\i + 1)}:\radius) -- ({360/\sides * \i}:\radius);
        \draw ({360/\sides * (\i + 4)}:\radius) -- ({360/\sides * \i}:\radius);
    \draw ({360/\sides * (\i + 12)}:\radius) -- ({360/\sides * \i}:\radius);
  }
      \foreach \i in {4,5,11} {
    \draw ({360/\sides * (\i + 1)}:\radius) -- ({360/\sides * \i}:\radius);
        \draw ({360/\sides * (\i + 9)}:\radius) -- ({360/\sides * \i}:\radius);
            \draw ({360/\sides * (\i + 12)}:\radius) -- ({360/\sides * \i}:\radius);
    \draw ({360/\sides * (\i + 12)}:\radius) -- ({360/\sides * \i}:\radius);
  }
\end{tikzpicture} \quad 
\begin{tikzpicture}[scale=0.28]
  \def\sides{13}
  \def\radius{3}

  \foreach \i in {1,...,\sides} {
    \fill ({360/\sides * \i}:\radius) circle (4pt);
  }
  
  \foreach \i in {1,3,5,7,10,12} {
    \draw ({360/\sides * (\i + 1)}:\radius) -- ({360/\sides * \i}:\radius);
        \draw ({360/\sides * (\i + 4)}:\radius) -- ({360/\sides * \i}:\radius);
    \draw ({360/\sides * (\i + 9)}:\radius) -- ({360/\sides * \i}:\radius);
    \draw ({360/\sides * (\i + 12)}:\radius) -- ({360/\sides * \i}:\radius);
  }
    \foreach \i in {2} {
    \draw ({360/\sides * (\i + 1)}:\radius) -- ({360/\sides * \i}:\radius);
    \draw ({360/\sides * (\i + 12)}:\radius) -- ({360/\sides * \i}:\radius);
  }
      \foreach \i in {4,9,11} {
    \draw ({360/\sides * (\i + 1)}:\radius) -- ({360/\sides * \i}:\radius);
        \draw ({360/\sides * (\i + 9)}:\radius) -- ({360/\sides * \i}:\radius);
            \draw ({360/\sides * (\i + 12)}:\radius) -- ({360/\sides * \i}:\radius);
  }
        \foreach \i in {6,8,13} {
    \draw ({360/\sides * (\i + 1)}:\radius) -- ({360/\sides * \i}:\radius);
        \draw ({360/\sides * (\i + 4)}:\radius) -- ({360/\sides * \i}:\radius);
            \draw ({360/\sides * (\i + 12)}:\radius) -- ({360/\sides * \i}:\radius);
  }
\end{tikzpicture} \quad
\begin{tikzpicture}[scale=0.28]
  \def\sides{13}
  \def\radius{3}

  \foreach \i in {1,...,\sides} {
    \fill ({360/\sides * \i}:\radius) circle (4pt);
  }
  
  \foreach \i in {1,3,5,7,9,10,12,13} {
    \draw ({360/\sides * (\i + 1)}:\radius) -- ({360/\sides * \i}:\radius);
        \draw ({360/\sides * (\i + 4)}:\radius) -- ({360/\sides * \i}:\radius);
    \draw ({360/\sides * (\i + 9)}:\radius) -- ({360/\sides * \i}:\radius);
    \draw ({360/\sides * (\i + 12)}:\radius) -- ({360/\sides * \i}:\radius);
  }
    \foreach \i in {2} {
    \draw ({360/\sides * (\i + 1)}:\radius) -- ({360/\sides * \i}:\radius);
    \draw ({360/\sides * (\i + 12)}:\radius) -- ({360/\sides * \i}:\radius);
  }
      \foreach \i in {4,11} {
    \draw ({360/\sides * (\i + 1)}:\radius) -- ({360/\sides * \i}:\radius);
        \draw ({360/\sides * (\i + 9)}:\radius) -- ({360/\sides * \i}:\radius);
            \draw ({360/\sides * (\i + 12)}:\radius) -- ({360/\sides * \i}:\radius);
  }
        \foreach \i in {6,8} {
    \draw ({360/\sides * (\i + 1)}:\radius) -- ({360/\sides * \i}:\radius);
        \draw ({360/\sides * (\i + 4)}:\radius) -- ({360/\sides * \i}:\radius);
            \draw ({360/\sides * (\i + 12)}:\radius) -- ({360/\sides * \i}:\radius);
  }
\end{tikzpicture} \quad
\begin{tikzpicture}[scale=0.28]
  \def\sides{13}
  \def\radius{3}

  \foreach \i in {1,...,\sides} {
    \fill ({360/\sides * \i}:\radius) circle (4pt);
  }
  
  \foreach \i in {1,3,4,5,7,8,9,10,12,13} {
    \draw ({360/\sides * (\i + 1)}:\radius) -- ({360/\sides * \i}:\radius);
        \draw ({360/\sides * (\i + 4)}:\radius) -- ({360/\sides * \i}:\radius);
    \draw ({360/\sides * (\i + 9)}:\radius) -- ({360/\sides * \i}:\radius);
    \draw ({360/\sides * (\i + 12)}:\radius) -- ({360/\sides * \i}:\radius);
  }
    \foreach \i in {2} {
    \draw ({360/\sides * (\i + 1)}:\radius) -- ({360/\sides * \i}:\radius);
    \draw ({360/\sides * (\i + 12)}:\radius) -- ({360/\sides * \i}:\radius);
  }
      \foreach \i in {6} {
    \draw ({360/\sides * (\i + 1)}:\radius) -- ({360/\sides * \i}:\radius);
        \draw ({360/\sides * (\i + 4)}:\radius) -- ({360/\sides * \i}:\radius);
            \draw ({360/\sides * (\i + 12)}:\radius) -- ({360/\sides * \i}:\radius);
  }
        \foreach \i in {11} {
    \draw ({360/\sides * (\i + 1)}:\radius) -- ({360/\sides * \i}:\radius);
        \draw ({360/\sides * (\i + 9)}:\radius) -- ({360/\sides * \i}:\radius);
            \draw ({360/\sides * (\i + 12)}:\radius) -- ({360/\sides * \i}:\radius);
  }
\end{tikzpicture}\quad
\begin{tikzpicture}[scale=0.28]
  \def\sides{13}
  \def\radius{3}

  \foreach \i in {1,...,\sides} {
    \fill ({360/\sides * \i}:\radius) circle (4pt);
  }
  
  \foreach \i in {1,3,5,7,9,10,12} {
    \draw ({360/\sides * (\i + 1)}:\radius) -- ({360/\sides * \i}:\radius);
        \draw ({360/\sides * (\i + 4)}:\radius) -- ({360/\sides * \i}:\radius);
    \draw ({360/\sides * (\i + 9)}:\radius) -- ({360/\sides * \i}:\radius);
    \draw ({360/\sides * (\i + 12)}:\radius) -- ({360/\sides * \i}:\radius);
  }
    \foreach \i in {2,4} {
    \draw ({360/\sides * (\i + 1)}:\radius) -- ({360/\sides * \i}:\radius);
    \draw ({360/\sides * (\i + 12)}:\radius) -- ({360/\sides * \i}:\radius);
  }
      \foreach \i in {6,8} {
    \draw ({360/\sides * (\i + 1)}:\radius) -- ({360/\sides * \i}:\radius);
        \draw ({360/\sides * (\i + 4)}:\radius) -- ({360/\sides * \i}:\radius);
            \draw ({360/\sides * (\i + 12)}:\radius) -- ({360/\sides * \i}:\radius);
  }
        \foreach \i in {11,13} {
    \draw ({360/\sides * (\i + 1)}:\radius) -- ({360/\sides * \i}:\radius);
        \draw ({360/\sides * (\i + 9)}:\radius) -- ({360/\sides * \i}:\radius);
            \draw ({360/\sides * (\i + 12)}:\radius) -- ({360/\sides * \i}:\radius);
  }
\end{tikzpicture}\quad
\begin{tikzpicture}[scale=0.28]
  \def\sides{13}
  \def\radius{3}

  \foreach \i in {1,...,\sides} {
    \fill ({360/\sides * \i}:\radius) circle (4pt);
  }
  
  \foreach \i in {1,3,5,7,8,9,10,12} {
    \draw ({360/\sides * (\i + 1)}:\radius) -- ({360/\sides * \i}:\radius);
        \draw ({360/\sides * (\i + 4)}:\radius) -- ({360/\sides * \i}:\radius);
    \draw ({360/\sides * (\i + 9)}:\radius) -- ({360/\sides * \i}:\radius);
    \draw ({360/\sides * (\i + 12)}:\radius) -- ({360/\sides * \i}:\radius);
  }
    \foreach \i in {2} {
    \draw ({360/\sides * (\i + 1)}:\radius) -- ({360/\sides * \i}:\radius);
    \draw ({360/\sides * (\i + 12)}:\radius) -- ({360/\sides * \i}:\radius);
  }
      \foreach \i in {4,6} {
    \draw ({360/\sides * (\i + 1)}:\radius) -- ({360/\sides * \i}:\radius);
        \draw ({360/\sides * (\i + 4)}:\radius) -- ({360/\sides * \i}:\radius);
            \draw ({360/\sides * (\i + 12)}:\radius) -- ({360/\sides * \i}:\radius);
  }
        \foreach \i in {11,13} {
    \draw ({360/\sides * (\i + 1)}:\radius) -- ({360/\sides * \i}:\radius);
        \draw ({360/\sides * (\i + 9)}:\radius) -- ({360/\sides * \i}:\radius);
            \draw ({360/\sides * (\i + 12)}:\radius) -- ({360/\sides * \i}:\radius);
  }
\end{tikzpicture}\quad
\begin{tikzpicture}[scale=0.28]
  \def\sides{13}
  \def\radius{3}

  \foreach \i in {1,...,\sides} {
    \fill ({360/\sides * \i}:\radius) circle (4pt);
  }
  
  \foreach \i in {1,3,4,5,7,9,10,11,13} {
    \draw ({360/\sides * (\i + 1)}:\radius) -- ({360/\sides * \i}:\radius);
        \draw ({360/\sides * (\i + 4)}:\radius) -- ({360/\sides * \i}:\radius);
    \draw ({360/\sides * (\i + 9)}:\radius) -- ({360/\sides * \i}:\radius);
    \draw ({360/\sides * (\i + 12)}:\radius) -- ({360/\sides * \i}:\radius);
  }
    \foreach \i in {2,8} {
    \draw ({360/\sides * (\i + 1)}:\radius) -- ({360/\sides * \i}:\radius);
        \draw ({360/\sides * (\i + 9)}:\radius) -- ({360/\sides * \i}:\radius);
    \draw ({360/\sides * (\i + 12)}:\radius) -- ({360/\sides * \i}:\radius);
  }
      \foreach \i in {6,12} {
    \draw ({360/\sides * (\i + 1)}:\radius) -- ({360/\sides * \i}:\radius);
        \draw ({360/\sides * (\i + 4)}:\radius) -- ({360/\sides * \i}:\radius);
            \draw ({360/\sides * (\i + 12)}:\radius) -- ({360/\sides * \i}:\radius);
  }
\end{tikzpicture}\quad
\begin{tikzpicture}[scale=0.28]
  \def\sides{13}
  \def\radius{3}

  \foreach \i in {1,...,\sides} {
    \fill ({360/\sides * \i}:\radius) circle (4pt);
  }
  
  \foreach \i in {1,3,4,5,7,8,9,10,11,12,13} {
    \draw ({360/\sides * (\i + 1)}:\radius) -- ({360/\sides * \i}:\radius);
        \draw ({360/\sides * (\i + 4)}:\radius) -- ({360/\sides * \i}:\radius);
    \draw ({360/\sides * (\i + 9)}:\radius) -- ({360/\sides * \i}:\radius);
    \draw ({360/\sides * (\i + 12)}:\radius) -- ({360/\sides * \i}:\radius);
  }
    \foreach \i in {2} {
    \draw ({360/\sides * (\i + 1)}:\radius) -- ({360/\sides * \i}:\radius);
        \draw ({360/\sides * (\i + 9)}:\radius) -- ({360/\sides * \i}:\radius);
    \draw ({360/\sides * (\i + 12)}:\radius) -- ({360/\sides * \i}:\radius);
  }
      \foreach \i in {6} {
    \draw ({360/\sides * (\i + 1)}:\radius) -- ({360/\sides * \i}:\radius);
        \draw ({360/\sides * (\i + 4)}:\radius) -- ({360/\sides * \i}:\radius);
            \draw ({360/\sides * (\i + 12)}:\radius) -- ({360/\sides * \i}:\radius);
  }
\end{tikzpicture}\quad
\begin{tikzpicture}[scale=0.28]
  \def\sides{13}
  \def\radius{3}

  \foreach \i in {1,...,\sides} {
    \fill ({360/\sides * \i}:\radius) circle (4pt);
  }
  
  \foreach \i in {1,3,5,6,7,8,9,10,12} {
    \draw ({360/\sides * (\i + 1)}:\radius) -- ({360/\sides * \i}:\radius);
        \draw ({360/\sides * (\i + 4)}:\radius) -- ({360/\sides * \i}:\radius);
    \draw ({360/\sides * (\i + 9)}:\radius) -- ({360/\sides * \i}:\radius);
    \draw ({360/\sides * (\i + 12)}:\radius) -- ({360/\sides * \i}:\radius);
  }
    \foreach \i in {2,4} {
    \draw ({360/\sides * (\i + 1)}:\radius) -- ({360/\sides * \i}:\radius);
        \draw ({360/\sides * (\i + 4)}:\radius) -- ({360/\sides * \i}:\radius);
    \draw ({360/\sides * (\i + 12)}:\radius) -- ({360/\sides * \i}:\radius);
  }
      \foreach \i in {11,13} {
    \draw ({360/\sides * (\i + 1)}:\radius) -- ({360/\sides * \i}:\radius);
        \draw ({360/\sides * (\i + 9)}:\radius) -- ({360/\sides * \i}:\radius);
            \draw ({360/\sides * (\i + 12)}:\radius) -- ({360/\sides * \i}:\radius);
  }
\end{tikzpicture}
\end{center}
\caption{All $P_8$-free minimal obstructions to $C_5$-coloring. The graphs in the first row are $P_6$-free. The graphs in the second row are $P_7$-free, but not $P_6$-free. All other graphs are $P_8$-free, but not $P_7$-free.}
\label{fig:p8FreeObstructions}
\end{figure}
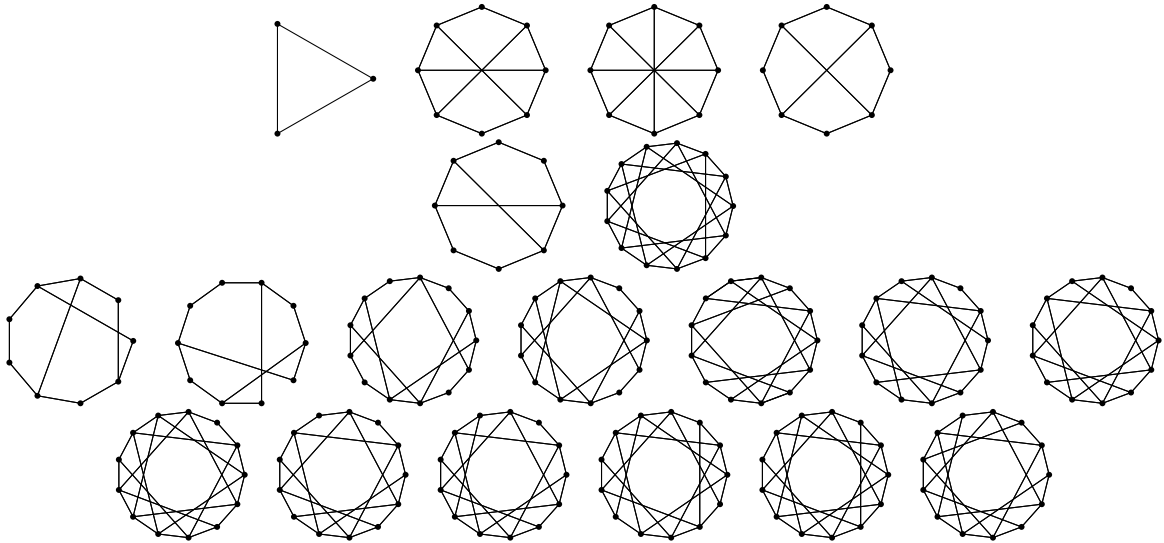

For the second step, we assume that our graph does not contain any $Q_i$, i.e., we consider graphs that are $\{P_8,K_3,Q_1,Q_2,Q_3,Q_4\}$-free. We show that such graphs are always $C_5$-colorable.
Consequently, each minimal obstruction to $C_5$-coloring (and, in general, every graph that is not $C_5$-colorable) was discovered in step 1. 

Before we discuss the second result, let us introduce the notation for subdivided claws. For integers $a,b,c \geq 1$, by $S_{a,b,c}$ we denote the graph obtained from the three-leaf star by subdividing each edge, respectively, $a-1$, $b-1$, and $c-1$ times.
Equivalently, $S_{a,b,c}$ is obtained from three paths $P_{a+1},P_{b+1},P_{c+1}$ by identifying one of their endpoints.

\begin{figure}[h!]
\begin{center}
\begin{tikzpicture}[scale=0.35]
  \def\sides{3}
  \def\radius{3}

  \foreach \i in {1,...,\sides} {
    \fill ({360/\sides * \i}:\radius) circle (4pt);
  }
  
  \foreach \i in {1,2,3} {
    \draw ({360/\sides * (\i + 1)}:\radius) -- ({360/\sides * \i}:\radius);
  }
\end{tikzpicture} \quad
\begin{tikzpicture}[scale=0.35]
  \def\sides{8}
  \def\radius{3}

  \foreach \i in {1,...,\sides} {
    \fill ({360/\sides * \i}:\radius) circle (4pt);
  }
  
  \foreach \i in {1,2,3,4,5,6,7,8} {
    \draw ({360/\sides * (\i + 1)}:\radius) -- ({360/\sides * \i}:\radius);
    \draw ({360/\sides * (\i + 4)}:\radius) -- ({360/\sides * \i}:\radius);
    \draw ({360/\sides * (\i + 7)}:\radius) -- ({360/\sides * \i}:\radius);
  }
\end{tikzpicture} \quad \\
\begin{tikzpicture}[scale=0.35]
  \def\sides{8}
  \def\radius{3}

  \foreach \i in {1,...,\sides} {
    \fill ({360/\sides * \i}:\radius) circle (4pt);
  }
  
  \foreach \i in {1,2,5,6} {
    \draw ({360/\sides * (\i + 1)}:\radius) -- ({360/\sides * \i}:\radius);
    \draw ({360/\sides * (\i + 7)}:\radius) -- ({360/\sides * \i}:\radius);
  }
  \foreach \i in {3,4,7,8} {
    \draw ({360/\sides * (\i + 1)}:\radius) -- ({360/\sides * \i}:\radius);
    \draw ({360/\sides * (\i + 4)}:\radius) -- ({360/\sides * \i}:\radius);
    \draw ({360/\sides * (\i + 7)}:\radius) -- ({360/\sides * \i}:\radius);
  }
\end{tikzpicture} \quad
\begin{tikzpicture}[scale=0.35]
  \def\sides{8}
  \def\radius{3}

  \foreach \i in {1,...,\sides} {
    \fill ({360/\sides * \i}:\radius) circle (4pt);
  }
  
  \foreach \i in {2,6} {
    \draw ({360/\sides * (\i + 1)}:\radius) -- ({360/\sides * \i}:\radius);
    \draw ({360/\sides * (\i + 7)}:\radius) -- ({360/\sides * \i}:\radius);
  }
  \foreach \i in {1,3,4,5,7,8} {
    \draw ({360/\sides * (\i + 1)}:\radius) -- ({360/\sides * \i}:\radius);
    \draw ({360/\sides * (\i + 4)}:\radius) -- ({360/\sides * \i}:\radius);
    \draw ({360/\sides * (\i + 7)}:\radius) -- ({360/\sides * \i}:\radius);
  }
\end{tikzpicture} \quad
\begin{tikzpicture}[scale=0.35]
  \def\sides{8}
  \def\radius{3}

  \foreach \i in {1,...,\sides} {
    \fill ({360/\sides * \i}:\radius) circle (4pt);
  }
  
  \foreach \i in {2,4,6,8} {
    \draw ({360/\sides * (\i + 1)}:\radius) -- ({360/\sides * \i}:\radius);
    \draw ({360/\sides * (\i + 7)}:\radius) -- ({360/\sides * \i}:\radius);
  }
  \foreach \i in {1,3,5,7} {
    \draw ({360/\sides * (\i + 1)}:\radius) -- ({360/\sides * \i}:\radius);
    \draw ({360/\sides * (\i + 4)}:\radius) -- ({360/\sides * \i}:\radius);
    \draw ({360/\sides * (\i + 7)}:\radius) -- ({360/\sides * \i}:\radius);
  }
\end{tikzpicture}
\begin{tikzpicture}[scale=0.35]
  \def\sides{13}
  \def\radius{3}

  \foreach \i in {1,...,\sides} {
    \fill ({360/\sides * \i}:\radius) circle (4pt);
  }
  
  \foreach \i in {1,2,3,4,5,6,7,8,9,10,11,12,13} {
    \draw ({360/\sides * (\i + 1)}:\radius) -- ({360/\sides * \i}:\radius);
    \draw ({360/\sides * (\i + 4)}:\radius) -- ({360/\sides * \i}:\radius);
    \draw ({360/\sides * (\i + 9)}:\radius) -- ({360/\sides * \i}:\radius);
    \draw ({360/\sides * (\i + 12)}:\radius) -- ({360/\sides * \i}:\radius);
  }
\end{tikzpicture} \quad
\end{center}
\caption{All $S_{2,2,1}$-free and all $S_{3,1,1}$-free minimal obstructions to $C_5$-coloring. The graphs in the first row are both $S_{2,2,1}$-free and $S_{3,1,1}$-free, whereas the first three graphs in the second row are $S_{3,1,1}$-free, but not $S_{2,2,1}$-free and the last graph in the second row is $S_{2,2,1}$-free, but not $S_{3,1,1}$-free.}
\label{fig:S221AndS311FreeObstructions}
\end{figure}
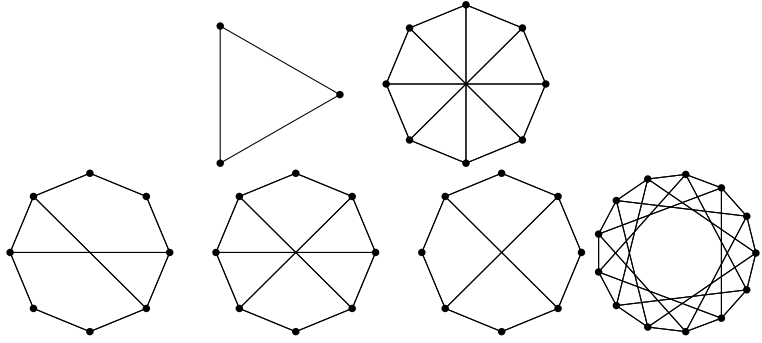

As our second result we show the following extension of the result of D\k{e}bski et al.~\cite{debski_et_al:LIPIcs.ISAAC.2022.14}.
  \begin{restatable}{theorem}{thmfiniteclaws}
  \label{thm:finite-claws}
    There are  3 minimal obstructions to $C_5$-coloring among $S_{2,2,1}$-free graphs, and
    5 minimal obstructions to $C_5$-coloring among $S_{3,1,1}$-free graphs.
  \end{restatable}
These graphs are shown in Figure~\ref{fig:S221AndS311FreeObstructions}. The proof is similar to the proof of Theorem~\ref{thm:finite-paths}.
First, we consider minimal obstructions that contain an induced $C_5$ and, using the computer search, we show that there is only a finite number of them.
Then, we show that graphs that exclude $K_3$ (as it is a minimal obstruction by itself), $C_5$, and also one of $S_{2,2,1}, S_{3,1,1}$, are either bipartite or are ``blown-up cycles'' -- in both cases $C_5$-colorability is straightforward to show.

\medskip

We complement these results with a construction of an infinite family of minimal obstructions to $C_5$-coloring.
\begin{restatable}{theorem}{thminfinite}
 \label{thm:infinite}
  There is an infinite family of minimal obstructions to $C_5$-coloring, which simultaneously exclude $P_{13}$, $S_{2,2,2}$,
$S_{5,5,1}$, $S_{11,1,1}$, and $S_{8,2,1}$ as an induced subgraph.
  \end{restatable}

The construction from Theorem~\ref{thm:infinite} is obtained by generalizing the infinite family of $P_7$-free minimal obstructions to 3-coloring, provided by Chudnovsky et al.~\cite{DBLP:journals/jct/ChudnovskyGSZ20}.
The idea can be further generalized. Let the \textit{odd girth} of $H$ be the length of a shortest odd cycle in $H$ (and keep it undefined for bipartite graphs).

\begin{restatable}{theorem}{thminfinitepathsgen}
 \label{thm:infinite-gen}
  Let $q \geq 3$ be odd, and let $H$ be a graph of odd girth $q$ that does not contain $C_4$ as a subgraph.
     There is an infinite family of minimal obstructions to $H$-coloring that are $\{P_{3q-1}, S_{2,2,2}, S_{\nicefrac{3(q-1)}{2}, \nicefrac{3(q-1)}{2}, 1}\}$-free.
  \end{restatable}

Note that Theorem~\ref{thm:infinite-gen} gives a bound for every graph $H$ that was discussed in (OR1) and (OR2).
An astute reader might notice that applying Theorem~\ref{thm:infinite-gen} for $H=C_5$, i.e., $q=5$,
yields a family of obstructions that are in particular  $P_{14}$-free and $S_{6,6,1}$-free, which does not match the bounds 
from Theorem~\ref{thm:infinite}.
Indeed, obtaining the refined result from Theorem~\ref{thm:infinite} requires some additional work, which again uses a mixture of combinatorial observations and computer search.

\paragraph{Organization of the paper.}
In Section~\ref{sec:prelim} we introduce some notation and preliminary observations.
In Section~\ref{sec:algo} we explain the algorithm that is later used to generate minimal obstructions.
In Section~\ref{sec:finite} we prove Theorem~\ref{thm:finite-paths}.
In Section~\ref{sec:finiteclaw} we prove Theorem~\ref{thm:finite-claws}.
In Section~\ref{sec:infinite} we provide constructions of infinite families of graphs that are then used to prove Theorems~\ref{thm:infinite} and~\ref{thm:infinite-gen}.
Finally, in the Appendix, we discuss implementation details of our algorithms and how they were tested for correctness.
 
\section{Preliminaries}\label{sec:prelim}
For an integer $n \geq 1$ we denote by $[n]$ the set $\{1,\ldots,n\}$, and by $[n]_0$ the set $[n] \cup \{0\}$.
For a graph $G=(V,E)$ and a vertex set $U \subseteq V$, the graph $G[U]$ denotes the subgraph of $G$ induced by $U$.
The graph $G-U$ denotes $G[V(G) \setminus U]$.
The set $N_G(u)$ denotes the neighborhood of vertex $u$ in the graph $G$.
For $U \subseteq V(G)$ we define $N_G(U)=\bigcup_{u \in U}N_G(u) \setminus U$.
If the graph $G$ is clear from the context, we omit the subscript and write $N(u)$ and $N(U)$.

If there exists an $H$-coloring of $G$, we denote this fact by $G \to H$. 
It is straightforward to verify that if $G \to H$, then $\textsf{odd-girth}(G) \geq \textsf{odd-girth}(H)$. In particular, $K_3$ has no $C_5$-coloring and is actually a minimal obstruction to $C_5$-coloring. Consequently, every other minimal obstruction to $C_5$-coloring is $K_3$-free.

For any two graphs $G$ and $H$ such that $G$ is $H$-colorable, the graph $\hull(G,H)$ denotes the graph with vertex set $V(G)$ and edge set $\{uv : u, v \in V(G)\text{ and for every } H\text{-coloring }c$ $\text{ of }G \text{, we have }c(u)c(v) \in E(H)\}$. Note that $\hull(G,H)$ is a supergraph of $G$ that is $H$-colorable and that for every induced subgraph $G'$ of $G$ we have $E(\hull(G',H)) \subseteq E(\hull(G,H))$. Note that $\hull(G,H)$ might contain some induced subgraphs that do not appear in $G$.

\section{Generating $F$-free minimal obstructions to $H$-coloring}\label{sec:algo}

In this section we describe an algorithm that can be used to generate all $F$-free minimal obstructions to $H$-coloring.
We emphasize that this approach is robust in the sense that it does not assume that $H=C_5$ and $F$ is a path or a subdivided claw, as needed for Theorems~\ref{thm:finite-paths} and~\ref{thm:finite-claws}.

Throughout the section graphs $H$ and $F$ are fixed. We will use the term \emph{minimal obstruction} for \emph{minimal obstruction to $H$-coloring}.
The algorithm takes as an input ``the current candidate graph'' $I$ that it tries to extend to a minimal obstruction by adding a new vertex $x$ and some edges between $x$ and $V(I)$. The precise details are given in Algorithm~\ref{algo:expand} (\lq Expand\rq) and are explained below. In particular, the algorithm can be used to generate all $F$-free minimal obstructions by choosing $I$ as the single-vertex graph. This algorithm is similar to the algorithm for \colo{k} by~\cite{DBLP:journals/jgt/GoedgebeurS18}, but we formulate it in a more general way.
 In case there are infinitely many minimal obstructions, the generation algorithm will never terminate.
 If there are only finitely many minimal obstructions, then the algorithm might still not terminate, since the prunning rules might not be strong enough. 
 However, if the algorithm terminates, it is guaranteed that there are only finitely many minimal obstructions and that the algorithm outputs all of them.

Let us explain the algorithm, see also the pseudocode in Algorithm~\ref{algo:expand}.
It starts from a graph $I$ and recursively expands this graph by adding a vertex and edges between this new vertex and existing vertices in each recursive step.
The expansion is based on expansion rules that aim at reducing the search space while ensuring that no minimal obstructions are lost in this operation.
For example, if an expansion leads to a graph $I'$ that is not $F$-free, the recursion can be stopped, because all further expansions of $I'$ will not be $F$-free either (note that we do not add any edges inside $V(I)$ and that the class of $F$-free graphs is hereditary).
Another way to restrict the search space is based on Lemma~\ref{similarGraphsLemma}.
This lemma and its proof follow Lemma 5 from~\cite{DBLP:journals/jct/ChudnovskyGSZ20} concerning \colo{k}, but generalizing it to \colo{H} required some adjustments.

\begin{lemma}
\label{similarGraphsLemma}
Let $G=(V,E)$ be a minimal obstruction to $H$-coloring and let $U$ and $W$ be two non-empty disjoint vertex subsets of $G$.
Let $J := \hull(G-U,H)$.
If there exists a homomorphism $\phi$ from $G[U]$ to $J[W]$, then there exists a vertex $u \in U$ for which $N_G(u) \setminus U \nsubseteq N_J(\phi(u))$.
\end{lemma}
\begin{proof}
Suppose that $N_G(u)\setminus U \subseteq N_J(\phi(u))$ for all $u \in U$.
Recall that $J$ admits an $H$-coloring, fix one and call it $c$.
Hence, for each $u \in U$ and each $v \in N_J(\phi(u))$ there is an edge $c(\phi(u))c(v) \in E(H)$.
Extend $c$ so that its domain is $V(G)$ by mapping any $u \in U$ to $c(\phi(u))$.
We now show that $c$ is an $H$-coloring of $G$. For contradiction, suppose there is an edge $uv \in V(G)$,
so that $c(u)c(v) \notin E(H)$. By symmetry, assume that $u \in U$, as $c$ restricted to $J$ is an $H$-coloring.

If $v \in U$, then $c(u)c(v) = c(\phi(u))c(\phi(v)) \in E(H)$ as, again, $c$ restricted to $J$ is an $H$-coloring.
Thus assume that $v \in N_G(u)\setminus U$.
Since $N_G(u)\setminus U \subseteq N_J(\phi(u))$, we obtain $c(u)c(v) = c(\phi(u))c(v) \in E(H)$.
This is a contradiction that $G$ is an obstruction to $H$-coloring.
\end{proof}

As isomorphism is a special type of a homomorphism, Lemma~\ref{similarGraphsLemma} immediately yields the following corollary.

\begin{corollary}\label{cor:comparableGraphs}
Let $G$ be an $H$-colorable graph that is an induced subgraph of a minimal obstruction $G'$.
Let $U, W \subseteq V(G)$ be two non-empty disjoint vertex subsets and let $J := \hull(G-U,H)$.
If there exists an isomorphism  $\phi : G[U] \rightarrow J[W]$ such that $N_G(u) \setminus U \subseteq N_J(\phi(u))$ for all $u \in U$, then there exists a vertex $x \in V(G') \setminus V(G)$ such that $x$ is adjacent to some vertex $u \in U$, but $x$ is not adjacent to $\phi(u)$ in $\hull(G'-U,H)$ (and thus also not adjacent to $\phi(u)$ in $\hull(G'[V(G) \cup \{x\} \setminus U],H)$).
\end{corollary}

Actually, we will only use the restricted version of Corollary~\ref{cor:comparableGraphs}. In what follows we use the notation and assumptions of the Corollary.
In case that $|U|=|W|=1$, say $U = \{u\}$ and $W = \{w\}$, we call the pair $(u,w)$ \textit{comparable vertices}.
In case that $G[U]$ and $J[W]$ are both isomorphic to $K_2$ and, say, $U = \{u,u'\}$ and $W = \{w,w'\}$,
we call the pair $(uu',ww')$ \textit{comparable edges}.
The algorithm concentrates on finding comparable vertices and edges for computational reasons.

\begin{algorithm}[ht!]
\caption{Expand}\label{algo:expand}
\DontPrintSemicolon
\LinesNumbered
\SetKwInOut{Input}{Input}
\SetKwInOut{Output}{Output}
\SetKwInOut{Constants}{Constants}
\SetKw{KwAnd}{and}
\SetKw{Print}{output}
\SetKwFunction{AlgName}{Expand}

\Constants{target graph $H$, forbidden graph $F$}
\Input{current graph $I$}
\Output{exhaustive list of $F$-free minimal obstructions to $H$-coloring}
\BlankLine            
  		\If{$I$ is $F$-free \KwAnd not generated before}
            {\label{line:isocheck}
			\If{$I$ is not $H$-colorable}
                {
    				\lIf{$I$ is a minimal obstruction to $H$-coloring}
                        {       
					   \Print $I$\label{line:k-critical}
                        }
                }
    		\Else
                {
        	       \If{$I$ contains comparable vertices $(u,v)$}
                      {
    		              \ForEach{graph $I'$ obtained from $I$ by adding a new vertex $x$ and edges between $x$ and vertices in $V(I)$ in all possible ways, such that $ux \in E(I')$, but $vx \notin E(\hull(I'-u,H))$}
                            {
                                \AlgName{$I'$}\label{line:dominatedvertex}
                            }
    		         }	
    		       \ElseIf{$I$ contains comparable edges $(uv,u'v')$}
                      {
                            \ForEach{graph $I'$ obtained from $I$ by adding a new vertex $x$ and edges between $x$ and vertices in $V(I)$ in all possible ways, such that $rx \in E(I')$, but $rx \notin E(\hull(I'-\{u,v\},H)))$ for some $r \in \{u,v\}$}
                            {
    			                 \AlgName{$I'$}\label{line:dominatededge}	
                            }
                      }
    		      \Else
                    {
    		              \ForEach{graph $I'$ obtained from $I$ by adding a new vertex $x$ and edges between $x$ and vertices in $V(I)$ in all possible ways}
                            {
    			             \AlgName{$I'$}
                            }
                    }
               }
               }
\end{algorithm}


We refer the interested reader to the Appendix for additional details about the efficient implementation of this algorithm, independent correctness verifications and sanity checks.

We conclude this section by giving some intuition about when the algorithm can be expected to terminate. In general, this is a really difficult question to answer as this is affected by both the choice of $H$, $F$ and $I$. If the current graph is not $F$-free, the algorithm will not expand it. Therefore, if the start graph $I$ contains a large induced subgraph of $F$ as an induced subgraph, this heavily restricts how $I$ can be expanded. Moreover, the algorithm will also not expand the current graph if it is not $H$-colorable. Therefore, the set of all $H$-colorings of the start graph $I$ seems important, since it restricts how $I$ can be expanded.

\section{Minimal obstructions to $C_5$-coloring with no long paths}\label{sec:finite}

In this section we still only discuss $C_5$-colorings, thus we will keep writing \emph{minimal obstructions} for \emph{minimal obstructions for $C_5$-coloring}.

The algorithm from Section~\ref{sec:algo} was implemented for $H=C_5$ (the source code is made publicly available at~\cite{gitrepo}). We used the algorithm, combined with some purely combinatorial observations, to generate an exhaustive list of $P_t$-free minimal obstructions, where $t \leq 8$; see also Figure~\ref{fig:p8FreeObstructions} and Table~\ref{tab:countsOfObstructions}.
The minimal obstructions can also be obtained from the database of interesting graphs at the \textit{House of Graphs}~\cite{DBLP:journals/dam/CoolsaetDG23} by searching for the keywords ``minimal obstruction to C5-coloring''.

\begin{table}[h!] \centering
	\begin{threeparttable}
		\begin{tabular}{cccccccccccccccccccccc} \\
			\hline
			\noalign{\smallskip}
			$n$\textbackslash$t$ & 6 & 7 & 8\\
			\noalign{\smallskip}
			\hline
			\noalign{\smallskip}
			\multicolumn{1}{c}{3} & 1 & 1 & 1\\
			\multicolumn{1}{c}{4} & 0 & 0 & 0\\
			\multicolumn{1}{c}{5} & 0 & 0 & 0\\
			\multicolumn{1}{c}{6} & 0 & 0 & 0\\
			\multicolumn{1}{c}{7} & 0 & 0 & 0\\
			\multicolumn{1}{c}{8} & 3 & 4 & 4\\
			\multicolumn{1}{c}{9} & 0 & 0 & 1\\
			\multicolumn{1}{c}{10} & 0 & 0 & 1\\
			\multicolumn{1}{c}{11} & 0 & 0 & 0\\
			\multicolumn{1}{c}{12} & 0 & 0 & 0\\
			\multicolumn{1}{c}{13} & 0 & 1 & 12\\
			\hline
			\multicolumn{1}{c}{Total} & 4 & 6 & 19\\
			\hline
		\end{tabular}
	\end{threeparttable}
	\caption{The number of $P_t$-free minimal obstructions for $C_5$-coloring.}	
		\label{tab:countsOfObstructions}
\end{table}

\subsection{$P_t$-free minimal obstructions for $t \in \{6,7\}$}
As a warm-up, let us reprove the result of Kami\'{n}ski and Pstrucha~\cite{DBLP:journals/dam/KaminskiP19} (in a slightly stronger form, as they did not provide the explicit list of minimal obstructions).
It will also serve as a demonstration of the way how Algorithm~\ref{algo:expand} is intended to be used. 

An exhaustive list for $t \leq 6$ can be obtained by running the algorithm from Section~\ref{sec:algo} with parameters $(I=K_1, H=C_5, F=P_6)$.\footnote{Let us remark that it is a simple exercise to find this list by hand.} This leads to the following observation.

\begin{observation}
There are four minimal obstructions for $C_5$-coloring among $P_6$-free graphs.
All of these obstructions, except for the triangle $K_3$, are $P_6$-free and not $P_5$-free.
\end{observation}

Unfortunately, the same simple strategy already fails for $t=7$.
Indeed, the algorithm as presented in Section~\ref{sec:algo} does not terminate after running for several hours after calling it with parameters $(I=K_1, H=C_5, F=P_7)$. However, with relatively small adjustments, the algorithm is able to produce an exhaustive list of minimal obstructions in a few seconds. 

The first adjustment has to do with the order in which the expansion rules are used. Note that the order in which the algorithm checks whether it can find comparable vertices and comparable edges does not affect the correctness of the algorithm, but it might affect whether the algorithm terminates or not. 
For example, it could happen that by expanding in order to get rid of a pair of comparable vertices, a new pair of comparable vertices is introduced (and this could continue indefinitely). By applying a different expansion rule first, this can be avoided sometimes. For $F=P_7$, the algorithm was run by first looking for comparable vertices and then for comparable edges, except when $|V(I)|=10$, in which case the algorithm first looks for comparable edges and then for comparable vertices.

The second adjustment is based on the following observation.

\begin{observation}\label{obs:p7start}
Every $P_7$-free minimal obstruction to $C_5$-coloring, except for the graph $K_3$, contains the cycle $C_5$ or the cycle $C_7$ as an induced subgraph.
\end{observation}
\begin{proof}
Let $G$ be $P_7$-free minimal obstruction such that $G \neq K_3$.
Note that $G$ is not bipartite, as every $K_2$-coloring (i.e., 2-coloring) of $G$ is also a $C_5$-coloring of $G$.
Hence, $G$ contains an odd cycle as an induced subgraph. Let $C$ be the shortest such cycle, and let $k = |V(C)|$.
Note that $k > 3$. Indeed, $C_3 = K_3$ is not $C_5$-colorable, so by minimality, $G$ cannot contain it as a proper subgraph, and $G$ is not a triangle itself by assumption.
On the other hand, any induced cycle on at least $8$ vertices contains $P_7$ as an induced subgraph.
Thus $k \in \{5,7\}$.
\end{proof}

By Observation~\ref{obs:p7start}, each $P_7$-free minimal obstruction belongs to at least one of the following three subsets:
\begin{enumerate}[(i)]
    \item the triangle $K_3$,
    \item minimal obstructions that are $P_7$-free but contain $C_5$ as an induced subgraph,
    \item minimal obstructions that are $P_7$-free but contain $C_7$ as an induced subgraph.
\end{enumerate}
Thus, running the algorithm for $(I=C_5,H=C_5,F=P_7)$ and $(I=C_7,H=C_5,F=P_7)$, respectively, we can generate the families (ii) and (iii). This yields the following result; recall that the finiteness of the family of minimal obstructions was already shown by Kami\'{n}ski and Pstrucha~\cite{DBLP:journals/dam/KaminskiP19}.

\begin{observation}
There are six $P_7$-free minimal obstructions to $C_5$-coloring.
Two of these obstructions are $P_7$-free, but not $P_6$-free.
\end{observation}

\subsection{$P_8$-free minimal obstructions}

This section is devoted to the proof of Theorem~\ref{thm:finite-paths}, which we restate below (see also Figure~\ref{fig:p8FreeObstructions}).

\thmfinitepaths*

Similarly to the $P_7$-free case, the proof uses the algorithm from Section~\ref{sec:algo},
but this time it requires a lot more purely combinatorial insight.
For $i \in [4]$, let $Q_i$ be the graph obtained from two disjoint copies of $C_5$ by identifying $i$ pairs of consecutive corresponding vertices of the cycles (see Figure~\ref{fig:qs}).

Let $G$ be a $P_8$-free minimal obstruction; we aim to understand the structure of $G$ and show that is must be one of 19 graphs in Figure~\ref{fig:p8FreeObstructions}.
We split the reasoning into two cases: first, we assume that $G$ contains $Q_i$, for some $i \in [4]$, as an induced subgraph.
Then, in the second case, we assume that $G$ is $\{Q_1,Q_2,Q_3,Q_4\}$-free.

\subsubsection*{Case 1: $G$ contains $Q_i$ for some $i \in [4]$ as an induced subgraph}

We deal with this case using the algorithm from Section~\ref{sec:algo}.
The algorithm terminates in a few minutes when it is called with the parameters $(I=Q_i,H=C_5,F=P_8)$ for all $i \in [4]$.
All minimal obstructions obtained this way are listed in Figure~\ref{fig:p8FreeObstructions}.

\subsubsection*{Case 2: $G$ does not contain $Q_i$ for any $i \in [4]$}

Note that as  $K_3$ is a minimal obstruction, from now on we will assume that $G$ is $\{P_8,K_3,Q_1,Q_2,Q_3,Q_4\}$-free.
We aim to show that all such graphs are $C_5$-colorable, i.e., the list obtained in Case 1, plus the triangle, is exhaustive.

\begin{lemma}
Let $G$ be a $\{P_8,K_3,Q_1,Q_2,Q_3,Q_4\}$-free graph. Then $G$ is $C_5$-colorable.
\end{lemma}

\begin{proof}
Without loss of generality we can assume that $G$ is connected, as otherwise we can treat each component separately.
Clearly, if $G$ is bipartite, then $G \to C_5$, thus we can assume that $G$ contains an odd cycle.
Consider a shortest such cycle in $G$; note that it is induced.
As $G$ is $\{P_8,K_3\}$-free, we observe that this cycle has either five or seven vertices.

Thus, we distinguish two cases: either $G$ contains an induced $C_5$ or $G$ contains an induced $C_7$ and is $C_5$-free.

\paragraph{Case 2A: $G$ has an induced $C_5$.}
Fix an induced copy of $C_5$ in $G$ and denote its consecutive vertices by $c_1,\ldots,c_5$. All computations on indices are performed modulo 5.
We also denote $C = \{c_1,\ldots,c_5\}$.
Let $A_i=N(c_i) \setminus \{c_{i-1},c_{i+1}\}$, and let $A=N(C) = \bigcup_{i = 1}^5 A_i$.

\begin{claim}\label{cla:a-neighbor}
For each $v \in A$ there exists a unique $i \in [5]$ such that $v \in A_i$.
Moreover, for every $i \in [5]$ and every $v \in A_i$, we have that $N(v) \cap A \subseteq A_{i-1} \cup A_{i+1}$. 
\end{claim}
\begin{claimproof}
First, we observe that if $v \in A_i \cap A_{i+1}$ for some $i \in [5]$, then $v,c_i,c_{i+1}$ is a triangle in $G$, a contradiction. 
Similarly, if $v \in A_i \cap A_{i+2}$, then the set $C \cup \{v\}$ induces a copy of $Q_4$ in $G$, a contradiction.
This proves the first statement of the claim.

For the second statement, consider $v \in A_i$ and $u \in N(v) \cap A$. 
If $u \in A_i$, then $v,u,c_i$ is again a triangle in $G$, a contradiction.
On the other hand, if $u \in A_{i+2}$ (or, by symmetry, $u \in A_{i-2}$), then $C \cup \{v,u\}$ induces a copy of $Q_3$ in $G$, a contradiction.
\end{claimproof}

By the claim above we obtain in particular that $(A_1,\ldots,A_5)$ is a partition of $A$.
Let $B = N(C \cup A)$ and $T = N(C \cup A \cup B)$. In other words, $B$ (resp. $T$) contains vertices at distance exactly 2 (resp. 3) from $C$.

We note that $V(G) = C \cup A \cup B \cup T$. Suppose otherwise. By connectivity of $G$, there is $v_1 \notin C \cup A \cup B \cup T$ with a neighbor $v_2 \in T$.
Let $v_3$ be a neighbor of $v_2$ in $B$, $v_4$ be a neighbor of $v_3$ in $A$, and $c_i$ be the unique neighbor of $v_4$ in $C$.
Then $(v_1,v_2,v_3,v_4,c_i, c_{i+1},c_{i+2},c_{i+3})$ is an induced $P_8$ in $G$, a contradiction.

\begin{claim}\label{cla:t-ind}
The set $T$ is independent.
\end{claim}
\begin{claimproof}
For contradiction assume otherwise, and let $u,v \in T$ be adjacent. By the definition of $T$, there exists $v_1 \in B \cap N(v)$ and $v_2 \in A \cap N(v_1)$.
Let $c_i$ be the unique neighbor of $v_2$ in $C$.
As $G$ is triangle-free, $uv_1 \notin E(G)$. Thus $(u,v,v_1,v_2,c_i,c_{i+1},c_{i+2},c_{i+3})$ is an induced $P_8$, a contradiction.
\end{claimproof}

\begin{claim}
For each $v \in B$ there exists a unique $i \in [5]$ such that $N(v) \cap A \subseteq A_{i-1} \cup A_{i+1}$.
\end{claim}
\begin{claimproof}
Observe that if there exists $i \in [5]$, $u \in A_i \cap N(v)$ and $u' \in A_{i+1} \cap N(v)$, then either $v,u,u'$ is a triangle or $C \cup \{v,u,u'\}$ induces $Q_2$ in $G$, a contradiction. Thus $N(v) \cap A$ intersects at most two sets $A_i, A_j$ for some distinct $i,j \in [5]$, and, moreover, $|i-j| \neq 1$. The claim follows.
\end{claimproof}

For each $i \in [5]$ define $B_i=\{v\in B~|~N(v) \cap A \subseteq A_i\}$ and $B_{i-1,i+1}=B_{i+1,i-1}=\{v\in B~|~N(v) \cap A_{i-1} \neq \emptyset, N(v) \cap A_{i+1}\neq \emptyset\}$. Let $\cB=\{B_i\}_{i \in [5]} \cup \{B_{i-1,i+1}\}_{i \in [5]}$.

\begin{claim}\label{cla:b-b-ind}
For each $i \in [5]$ the set $B_i \cup B_{i,i-2} \cup B_{i,i+2}$ is independent. 
\end{claim}
\begin{claimproof}
If $u,v \in B_i \cup B_{i,i-2} \cup B_{i,i+2}$, then there exist $u'$ and $v'$ in $A_i$ such that $uu',vv' \in E(G)$. If now $uv \in E(G)$, then either $u,v,u'$ is a triangle in $G$ (if $u'=v'$) or $C \cup \{u,v,u',v'\}$ induces a copy of $Q_1$ in $G$, in both cases we reach a contradiction. Thus $B_i \cup B_{i,i-2} \cup B_{i,i+2}$ is an independent set.
\end{claimproof}



\begin{claim}\label{cla:p8-found}
For each $i \in [5]$, if $ub \in V(G)$ is such that $u \in T \cup B$ has no neighbors in $A_{i-2}$ (resp. $A_{i+2}$) and $b \in B_i \cup B_{i,i+2}$ (resp. $b \in B_i \cup B_{i,i-2}$), then $A_{i-2}=\emptyset$ (resp. $A_{i+2}=\emptyset$).
\end{claim}
\begin{claimproof}
Without loss of generality we can assume $i=1$. Clearly, it is enough to prove one statement, the other one will follow by symmetry. Let $a_1 \in A_1$ be a neighbor of $b$.
Observe that if $a_4 \in A_4$ then $(u,b,a_1,c_1,c_2,c_3,c_4,a_4)$ is a $P_8$ in $G$, a contradiction. 
    Indeed, $a_1a_4 \notin E(G)$ by Claim~\ref{cla:a-neighbor}, by assumption $u$ is non-adjacent to $a_4$, and non-adjacent to $a_1$ since $G$ is $K_3$-free. 
    The (non-)existence of all the other edges follows from the definitions of the vertices/the sets they belong to.
\end{claimproof}

\begin{claim}\label{cla:t-neighbors}
Every $t \in T$ has neighbors in at most one of the sets in $\cB$.
\end{claim}
\begin{claimproof}
First, suppose that $t$ has a neighbor $b_1$ in $B_i$ for $i \in [5]$, by symmetry we can assume that $i=1$.
For contradiction suppose that $b'$ is a neighbor of $t$ in $B \setminus B_1$, and let $a_1$ be a neighbor of $b_1$ in $A_1$.
As $G$ is triangle-free, we have $b_1b' \notin E(G)$.
Now if $a_1b' \notin E(G)$, then $(b',t,b_1,a_1,c_1,c_2,c_3,c_4)$ is an induced $P_8$ in $G$.
So assume that $a_1b' \in E(G)$. In particular, $b' \in B_{1,3}$ or $b' \in B_{1,4}$. By symmetry assume the former and let $a_3 \in N(b') \cap A_3$.
Then $(b_1,t,b',a_3,c_3,c_4,c_5,c_1)$ is an induced $P_8$ in $G$. In both cases we reach a contradiction.

So from now on we assume that $t$ has no neighbor in $\bigcup_{i=1}^5 B_i$, but it must have a neighbor, say, $b \in B_{2,5}$.
From Claim~\ref{cla:p8-found} (for $i=5$ and $i=2$) it follows that
$A_3 = A_4 = \emptyset$, so in particular $B_{1,3} \cup B_{2,4} \cup B_{1,4} \cup B_{3,5} = \emptyset$.
This means that $B = B_{2,5} \cup \bigcup_{i=1}^5 B_i$, so indeed $t$ has no neighbors in $B$ outside $B_{2,5}$.
\end{claimproof}

We obtain the following.
\begin{claim}\label{cla:t-homo}
If $h:G - T \to C_5$ is a homomorphism, in which for every $\widetilde{B} \in \cB$ there is $i \in [5]$
such that $h(\widetilde{B}) = i$,
then $h$ can be extended to a homomorphism  from $G$ to $C_5$.
\end{claim}
\begin{claimproof}
Note that by Claim~\ref{cla:t-neighbors} for every $t \in T$ there is $i \in [5]$ such that all neighbors of $t$ are mapped to $i$. Then we can map such $t$ to $i+1$; this will not create any wrongly mapped edge as $T$ is independent by Claim~\ref{cla:t-ind}.
\end{claimproof}

\paragraph{Subcase 2A.1: There exists $i \in [5]$ and $b_i \in B_i, b_{i+2} \in B_{i+2}$ such that $b_ib_{i+2} \in E(G)$.}
By symmetry assume that $i=1$, and for $j \in \{1,3\}$, let $a_j$ be a vertex in $N(b_i) \cap A_i$.
By Claim~\ref{cla:p8-found} (for $b=b_1,u=b_3$ and $i=1$, and for $b=b_3,u=b_1$ and $i=3$) we have $A_4=A_5=\emptyset$. Consequently, $B = B_1 \cup B_2 \cup B_3 \cup B_{1,3}$.

\begin{claim}
    There are no edges between $B_2$ and $B_{1,3} \cup B_1 \cup B_3$.
\end{claim}
\begin{claimproof}
Assume otherwise, and let $b_2b$ be an edge such that $b_2 \in B_2$ and $b \in B_{1,3} \cup B_1 \cup B_3$.
If $b \in B_{1} \cup B_3$, Claim~\ref{cla:p8-found} gives us that $A_3=\emptyset$ or $A_1=\emptyset$, a contradiction.
Thus $b \in B_{1,3}$.
Choose any $a'_3 \in A_3 \cap N(b)$, $a'_1 \in A_1 \cap N(b)$, if possible selecting $a'_1$ and $a'_3$ to be the neighbors of, respectively, $b_1$ and $b_3$. Let $a_2 \in A_2 \cap N(b_2)$.

Assume that $a'_1 \in N(b_1)$ and $a'_3 \in N(b_3)$. If $a'_1a_2 \notin E(G)$ (or, by symmetry, $a'_3a_2 \notin E(G)$) then $(a'_1,b,b_2,a_2,c_2,c_3,c_4,c_5)$ is an induced $P_8$ in $G$, a contradiction. Thus $a'_1a_2,a'_3a_2 \in E(G)$. But now $\{a'_1,b,a'_3,b_3,b_1,a_2\}$ induces $Q_4$ in $G$, a contradiction. Now, if $a'_1 \notin N(b_1)$ or $a'_3 \notin N(b_{3})$, by symmetry we assume that the first one holds. Then $(a'_1,c_1,a_1,b_1,b_3,a_3,c_3,c_4)$ is an induced $P_8$ in $G$ a contradiction.
\end{claimproof}

Now it is straightforward to verify that the function $h : V(G) \setminus T \to [5]$ defined as
\begin{align*}
h(x)=\begin{cases}
i & \textrm{if }x=c_i \textrm{ for }i\in [5] \textrm{ or }x\in B_i \textrm{ for }i=1, \\
i+1 & \textrm{if }x\in A_i \textrm{ for }i\in [3], \\
i+2 & \textrm{if }x\in B_i\textrm{ for }i\in \{2,3\}, \\ 
3 & \textrm{if }x\in B_{1,3},
\end{cases}
\end{align*}
is a homomorphism from $G - T$ to $C_5$. By Claim~\ref{cla:t-homo}, we have $G \to C_5$ (see Figure~\ref{fig:lemma2-case1a} (left)).


\paragraph{Subcase 2A.2: For each $i \in [5]$ the set $B_i \cup B_{i+2}$ is independent, but there exists $j \in [5]$ and $b_j \in B_j, b \in B_{j-1,j+2} \cup B_{j+1,j-2}$ such that $b_jb \in E(G)$. }
By symmetry assume that $j=1$ and $b\in B_{j+1,j-2}=B_{2,4}$.
By Claim~\ref{cla:p8-found} for $i=2$ we have $A_5=\emptyset$. 

\begin{claim}
    $A_3=\emptyset$ and $B_{4,1}=\emptyset$.
\end{claim}
\begin{claimproof}
Assume that there is $a_3 \in A_3$. Note that depending on whether $a_3a_4$ is an edge of $G$ or not, either $(b_{1},b,a_4,a_3,c_3,c_2,c_1,c_5)$ or $(c_1,a_1,b_1,b,a_4,c_4,c_3,a_3)$ is an induced $P_8$ in $G$. Thus $A_3=\emptyset$.

Now assume that there is $b_{4,1} \in B_{4,1}$. Choose any $a'_1 \in A_1 \cap N(b_{4,1})$, $a'_4 \in A_4 \cap N(b_{4,1})$, if possible selecting $a'_1$ and $a'_4$ to be the neighbors of, respectively, $b_1$ and $b$.

If $a'_1 \in N(b_1)$ and $a'_1a_2 \notin E(G)$, then $(c_4,c_3,c_2,a_2,b,b_1,a'_1,b_{1,4})$ induces $P_8$.
If $a'_1 \in N(b_1)$ and $a'_1a_2 \in E(G)$, then either $\{a_1',a_2,b,a'_4,b_{4,1},b_{1}\}$ induces $Q_4$ or $(c_1,c_2,a_2,b,a_4,c_4,$ $a'_4,b_{1,4})$ induces $P_8$, depending on whether $a'_4 \in N(b)$ or not (recall that $bb_{4,1} \notin E(G)$ by Claim~\ref{cla:b-b-ind}).

Thus we can assume that $a'_1 \notin N(b_1)$. Then $(a_1',c_1,a_1,b_1,b,a_4,c_4,c_3)$ induces $P_8$. Since in every case we reach a contradiction, the claim follows.
\end{claimproof}

From the fact that $A_5=A_3=B_{4,1}=\emptyset$ we obtain that $B=B_1 \cup B_2 \cup B_4 \cup B_{2,4}$.
Now it is straightforward to verify that the function $h : V(G) \setminus T \to [5]$ defined as
\begin{align*}
h(x)=\begin{cases}
i & \textrm{if }x=c_i \textrm{ for }i\in [5]  \textrm{ or }x\in B_{i-2,i} \textrm{ for }i=4, \\
i+1 & \textrm{if }x\in A_i \textrm{ for }i\in \{1,2,4\}, \\
i+2 & \textrm{if }x\in B_i\textrm{ for }i\in \{1,2,4\},
\end{cases}
\end{align*}
is a homomorphism from $G - T$ to $C_5$. By Claim~\ref{cla:t-homo}, we have $G \to C_5$ (see Figure~\ref{fig:lemma2-case1a} (right)).

\begin{figure}
    \centering
    \includegraphics[page=1,scale=1]{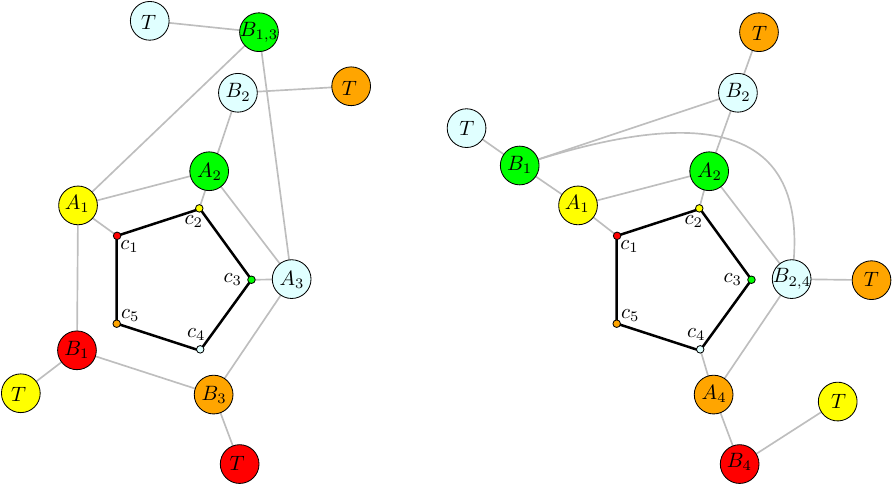}
    \caption{$C_5$-coloring of $G$ defined as in Subcase A1 (left) and Subcase A2 (right). Colors correspond to vertices/sets being mapped to specific vertices of $C_5$. All sets represented by the circles are independent. Two sets are joined by a gray line if there may be edges between them.}
    \label{fig:lemma2-case1a}
\end{figure}

\paragraph{Subcase 2A.3: For each $i \in [5]$ there is no edge between $B_i$ and the set $B_{i-2} \cup B_{i-1,i+2} \cup B_{i+2} \cup B_{i+1,i-2}$.}

Now we define $h: V(G) \setminus T \to [5]$ as follows (see Figure~\ref{fig:lemma2-case1b}).\begin{align*}
h(x)=\begin{cases}
i & \textrm{if }x=c_i \textrm{ or } x \in B_{i}\textrm{ for }i\in [5], \\
i+1 & \textrm{if }x\in A_i\textrm{ or } x \in B_{i-1,i+1} \textrm{ for }i\in [5].
\end{cases}
\end{align*}
Clearly, $h$ is a homomorphism from $G - T$ to $C_5$. Thus, by Claim~\ref{cla:t-homo}, $G \to C_5$. The lemma follows.

\begin{figure}
    \centering
    \includegraphics[page=3,scale=1.4]{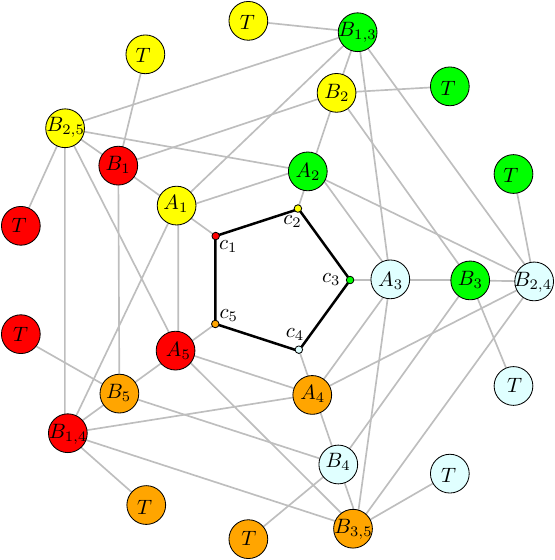}
    \caption{$C_5$-coloring of $G$ defined as in Subcase A3.}
    \label{fig:lemma2-case1b}
\end{figure}

\paragraph{Case 2B: $G$ is $C_5$-free.} Since $G$ is not bipartite, $G$ contains $C_7$ as an induced subgraph.
Let $C$ be the vertex set of some induced $C_7$ in $G$, denote the consecutive vertices as $c_1,\ldots,c_7$. Again, all operations on indices are performed modulo 7.
Let $A=N(C)$.
For $i\in [7]$ let $A_i=\{v \in A~|~N(v) \cap C=\{c_i\}\}$, $A_{i-1,i+1}=A_{i+1,i-1}=\{v \in A~|~N(v) \cap C=\{c_{i-1},c_{i+1}\}\}$, $\cA=\{A_i\}_{i \in [7]} \cup \{A_{i-1,i+1}\}_{i \in [7]}$.
Clearly, since $G$ is triangle-free, each set in $\cA$ is independent.
\begin{claim}
$\cA$ is a partition of $A$.
\end{claim}
\begin{claimproof}
Consider a vertex $v \in A$.
If $v$ has more than two neighbors in $C$, we either create a triangle or an induced $C_5$, a contradiction.
Moreover, because of that, if $v \in A$ has exactly two neighbors in $C$, they must be of the form $c_{i-1},c_{i+1}$ for some $i \in [7]$.
\end{claimproof}

\begin{claim}
For every $i \in [7]$ and every $v \in A_{i-1,i+1}$ we have $N(v) \cap A \subseteq A_{i-2,i} \cup A_{i,i+2} \cup A_i \cup A_{i-2} \cup A_{i+2}$.
\end{claim}
\begin{claimproof}
Recall that for each $i \in [7]$ the set $A_{i-1,i+1}$ is independent.
By symmetry consider $i=1$. Thus we need to show that $N(v) \cap A \subseteq A_{6,1} \cup A_{1,3} \cup A_1 \cup A_{6} \cup A_{3}$.
If there exists $u \in A_{2,4} \cup A_2$ (or, by symmetry, $u\in A_{5,7} \cup A_7$) such that $vu \in E(G)$, then $u,v,c_{2}$ is a triangle in $G$, a contradiction.
Next, if there exists $u \in A_{3,5}\cup A_5$ (or, by symmetry, $u\in A_{4,6} \cup A_4$) such that $vu \in E(G)$, then $\{u,v,c_7,c_6,c_5\}$ induces $C_5$ in $G$, also a contradiction.
\end{claimproof}

\begin{claim}
For each $i \in [7]$ and each $v \in A_i$ we either have $N(v) \cap A \subseteq A_{i-3} \cup A_{i-1,i+1} \cup A_{i-3,i-1}$ or $N(v) \cap A \subseteq A_{i+3}\cup A_{i-1,i+1} \cup A_{i+1,i+3}$.
\end{claim}
\begin{claimproof}
By symmetry assume $i=1$, and thus we have to show that $N(v) \cap A \subseteq A_{5} \cup A_{7,2} \cup A_{5,7}$ or $N(v) \cap A \subseteq A_{4}\cup A_{7,2} \cup A_{2,4}$. Clearly, $N(v) \cap A_1=N(v) \cap A_{1,3} = N(v) \cap A_{6,1} = \emptyset$, since $G$ is triangle-free.

Note that if there exists $u \in A_{2}$ (or, by symmetry, $u\in A_{7}$) such that $vu \in E(G)$, then $(u,v,c_1,c_7,c_6,c_5,c_4,c_3)$ is an induced $P_8$ in $G$, a contradiction. 
Moreover, if there exists $u \in A_{3}$ (or, by symmetry, $u\in A_{6}$) such that $vu \in E(G)$, then $\{u,v,c_1,c_2,c_3\}$ induces $C_5$ in $G$, a contradiction.
And if there is $u \in A_{3,5}$ (or, by symmetry, $u\in A_{4,6}$) such that $vu \in E(G)$, then $\{v,u,c_3,c_2,c_1\}$ induces $C_5$ in $G$, a contradiction.

Last, assume that there exists $u \in N(v) \cap (A_{5} \cup A_{5,7})$. We need to show that $u$ has no neighbors in $A_{4} \cup A_{2,4}$; by symmetry this will conclude the proof of claim.
If there is $w \in N(v) \cap (A_{4} \cup A_{2,4})$, then $\{u,v,w,c_4,c_5\}$ induces $C_5$ in $G$, a contradiction.
\end{claimproof}

Thus for every $i\in [7]$ we can further partition $A_i$ into $A^-_i=\{v \in A_i~|~N(v) \cap (A_{i+3}\cup A_{i+1,i+3})=\emptyset\}$, and $A^+_i=\{v \in A_i~|~N(v) \cap (A_{i+3}\cup A_{i+1,i+3})\neq\emptyset\}$.
Let $B = V(G) \setminus (C \cup A)$.

\begin{claim}
The set $B$ is independent, and for each $v \in B$ there is $i \in [7]$ such that $N(v) \subseteq A_{i-2,i} \cup A_{i,i+2}$.
\end{claim}
\begin{claimproof}
Let $v \in B$.
First, we note that for every $i \in [7]$ we have $N(v) \cap A_i = \emptyset$. Indeed, otherwise there exists $v' \in N(v) \cap A_i$ and $(v,v',c_i,c_{i+1},c_{i+2},c_{i+3},c_{i-3},c_{i-2})$ is a $P_8$ in $G$, a contradiction.

Next, note that $N(v) \cap A \neq \emptyset$. Otherwise, as $G$ is connected, there is $i \in [7]$, $a \in A_{i-1,i+1}$, and an induced path $(v,v_1,\ldots,v_p,a)$ in $G-C$ for some $p \geq 1$. Then $(v,v_1,\ldots,v_p,a,c_{i+1},c_{i+2},$ $c_{i+3},c_{i-3},c_{i-2})$ is an induced path on at least 8 vertices, a contradiction.

Therefore, we can assume that there exists $i \in [7]$ and $v' \in A_{i-2,i} \cap N(v)$. Without loss of generality, let $i=1$.
If now there is $u \in B \cap N(v)$, then clearly $uv' \notin E(G)$, as $G$ is triangle-free. Then $(u,v,v',c_1,c_2,c_3,c_4,c_5)$ is a $P_8$ in $G$, a contradiction. We conclude that $B$ is independent.

Now suppose that $v$ has a neighbor $u \in A_{j,j+2}$ for some $j \notin \{1,4,6\}$. If $j=7$ (or, by symmetry, $j=5$), then $\{v,v',c_1,c_7,u\}$ induces a $C_5$ in $G$, a contradiction. On the other hand, if $j=2$ (or, by symmetry, $j=3$), then $\{v',v,u,c_2,c_1\}$ is a $C_5$ in $G$, a contradiction.

Last, suppose that $v$ has neighbors $u \in A_{4,6}$ and $u' \in A_{1,3}$.
Then $\{u,v,u',c_3,c_4\}$ is an induced $C_5$, a contradiction.
\end{claimproof}

We can define now a function $h: V(G)\setminus B \to [5]$:
\begin{align*}
h(x)=\begin{cases}
i & \textrm{if }x\in\{c_i\}\cup A_{i-1,i+1}\textrm{ for }i \in [5], \\
i+1& \textrm{if }x\in A^+_{i}\textrm{ for }i \in \{0,...,4\} \\
i-1 & \textrm{if }x\in A^-_{i}\textrm{ for }i \in \{2,...,6\}, \\
1& \textrm{if }x\in\{c_6\} \cup A^+_5 \cup A_{5,7}\cup A^-_7, \\
2& \textrm{if }x\in\{c_7\} \cup A^-_1 \cup A_{6,1}\cup A^+_6.
\end{cases}
\end{align*}
It is now straightforward to verify that $h$ is a homomorphism from $G - B$ to $C_5$. Moreover, for each $i \in [7]$ the set $A_{i-1,i+1}$ is mapped by $h$ to precisely one vertex of $C_5$. Thus, analogously as in Claim~\ref{cla:t-homo}, $h$ can be extended to a homomorphism from $G$ to $C_5$ (see Figure~\ref{fig:lemma2-case2}). This concludes the proof of the lemma.
\end{proof}

\begin{figure}
    \centering
    \includegraphics[page=2,scale=1.4]{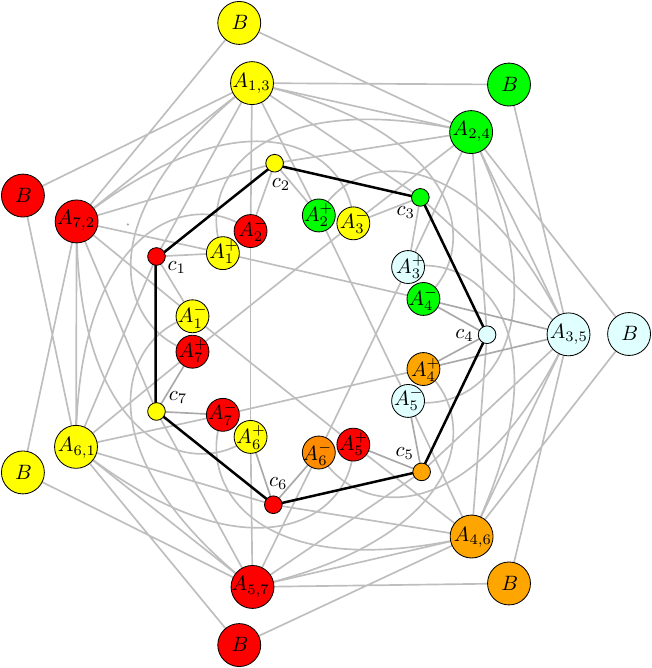}
    \caption{$C_5$-coloring of $G$ defined as in Case 2B.}
    \label{fig:lemma2-case2}
\end{figure}

\section{Minimal obstructions to $C_5$-coloring with no long subdivided claws}\label{sec:finiteclaw}
This section is devoted to the proof of Theorem~\ref{thm:finite-claws}.

\thmfiniteclaws*

The minimal obstructions are shown in Figure~\ref{fig:S221AndS311FreeObstructions} and can also be obtained from the database of interesting graphs at the \textit{House of Graphs}~\cite{DBLP:journals/dam/CoolsaetDG23} by searching for the keywords ``S221-free minimal obstruction to C5-coloring'' and ``S311-free minimal obstruction to C5-coloring'', respectively. As in this section the target graph is always $C_5$, we will keep writing \emph{minimal obstructions} for \emph{minimal obstructions to $C_5$-coloring}.
We proceed similarly as in the proof of Theorem~\ref{thm:finite-paths}.
Let $G$ be an $S_{2,2,1}$-free (respectively, $S_{3,1,1}$-free) minimal obstruction.
We again consider two cases.

\subsubsection*{Case 1: $G$ contains $C_5$ as an induced subgraph.}
This case is solved using the algorithm from Section~\ref{sec:algo}.
The algorithm is called with the parameters $(I=C_5,H=C_5,F=S_{2,2,1})$ and then with parameters $(I=C_5,H=C_5,F=S_{3,1,1})$. Both calls terminate,  returning a finite list of minimal obstructions.

\subsubsection*{Case 2: $G$ does not contain $C_5$ as an induced subgraph.}
Similarly as in the proof of Theorem~\ref{thm:finite-paths}, note that $K_3$ is a minimal $\{F,C_5\}$-free obstruction for $F \in \{S_{2,2,1},S_{3,1,1}\}$.
Thus, from now on, we assume that $G$ is $K_3$-free and prove that there are no more minimal obstructions satisfying this case, i.e., the following result.

\begin{lemma}\label{lem:sttt-finite}
Let $F \in \{S_{2,2,1},S_{3,1,1}\}$ and let $G$ be a $\{F,C_5,K_3\}$-free graph. Then $G$ is $C_5$-colorable.
\end{lemma}
\begin{proof}
We can assume that $G$ is connected, as otherwise we can consider each of its connected components separately.
Clearly, if $G$ is bipartite, then the lemma holds, thus we can assume that there is an odd cycle $C$ of minimal length in $G$. Let $t \geq 7$ be the number of vertices of $C$, i.e., $C=C_t$. Denote by $c_0,c_1,c_2,\ldots,c_{t-1}$ the consecutive vertices of $C$. Since $t$ is minimal, $C$ must be an induced cycle.

Note that every $v \in V(G) \cap (N(C) \cup C)$ has precisely two neighbors $c_{i-1},c_{i+1}$ on $C$ for some $i \in [t-1]_0$ (we compute the indices modulo $t$). Indeed, if $v$ has one neighbor on $C$, say $c_i$, then $\{c_{i-2},c_{i-1},c_i,c_{i+1},c_{i+2},v\}$ induces $S_{2,2,1}$ in $G$, and $\{c_{i-3},c_{i-2},c_{i-1},c_i,c_{i+1},v\}$ induces $S_{3,1,1}$ in $G$, contradicting the fact that $G$ is $F$-free.
On the other hand, if $v$ has two neighbors $c_i,c_j$ such that $|i-j|\neq 2$ (in particular if $|N(v) \cap C| \geq 3$), then either $v,c_i,c_{i+1},\ldots,c_{j-1},c_j$ or $v,c_j,c_{j+1},\ldots,c_{i-1},c_i$ is an odd cycle on less than $t$ vertices, a contradiction with the minimality of $C$.

Now observe that $V(G)=N(C) \cup C$. Indeed, otherwise, since $G$ is connected, there is a vertex $v \in N(c_{i-1})\cup N(c_{i+1})$, and a vertex $u \in N(v) \setminus (N(C)\cup C)$. But then $\{c_{i-2},c_{i-1},v,c_{i+1},c_{i+2},u\}$ induces $S_{2,2,1}$ in $G$, and $\{c_{i-3},c_{i-2},c_{i-1},v,c_{i+1},u\}$ induces $S_{3,1,1}$ in $G$, again contradicting the fact that $G$ is $F$-free.

Thus, we can partition the vertices of $G$ into sets $A_{1,3},A_{2,4},\ldots,A_{t-2,1},A_{t-1,2}$ so that for each $v \in A_{i-1,i+1}$ we have $N(v) \cap C=\{c_{i-1},c_{i+1}\}$. Note that for every $i \in [t-1]_0$ and every $v \in A_{i-1,i+1}$ we have $N(v) \subseteq A_{i-2,i} \cup A_{i,i+2}$. Indeed, otherwise we would again obtain a shorter odd cycle in $G$, a contradiction.

Now it is straightforward to verify that a function $h:V(G) \to V(C_5)$ defined as:
\begin{align*}
h(v)=\begin{cases}
i & \textrm{ if }i \in \{0,1,2\} \textrm{ and } v \in A_{i-1,i+1}, \\
4  & \textrm{ if }i \geq 3 \textrm{ is odd and } v \in A_{i-1,i+1}, \\
5  & \textrm{ if }i \geq 3 \textrm{ is even and } v \in A_{i-1,i+1},
\end{cases}
\end{align*}
is a $C_5$-coloring of $G$.
\end{proof}

Combining the cases, we obtain the statement of Theorem~\ref{thm:finite-claws}.

\section{An infinite family of minimal obstructions}\label{sec:infinite}
In this section we construct infinite families of graphs that will be later used to prove Theorems~\ref{thm:infinite} and~\ref{thm:infinite-gen}. The construction is a generalization of the one designed for \colo{3}~\cite{DBLP:journals/jct/ChudnovskyGSZ20}; the authors attribute it to Pokrovskiy.

\subsection{The construction}

For every odd $q \geq 3$ and every $p \geq 1$, let $G_{q,p}$ be the graph on vertex set $[qp-3]_0$ (all arithmetic operations on $[qp-3]_0$ here are done modulo $qp-2$), such that for every $i \in [qp-3]_0$ it holds that
\[N(i)=\{{i-1},{i+1}\} \cup \{{i+qj-1} ~|~ j \in [p-1] \}. \]

To simplify the arguments, we partition the vertex set of $G_{q,p}$ into $q$ sets $V_s=\{i~|~i = s \mod q\}$, where $s \in [q-1]_0$. For a vertex $i \in V(G_{q,p})$ and $s \in [q-1]_0$, we say that $i$ is of \emph{type $s$}, if $i \in V_{s}$.
If we say that $i$ is of the type $s$ for some $s \geq q$, it means that $i$ is of type $s \mod q$.
If $i,j \in V(G_{q,p})$, we say that $(i,j)$ is \emph{of type} $(s_i,s_j)$ if $i$ is of type $s_i$ and $j$ is of type $s_j$.

The following observation follows immediately from the definition of $E(G_{q,p})$.
\begin{observation}\label{obs:neighbors}
Let $q \geq 3$ be odd, and let $ij \in E(G_{q,p})$ be such that $i \in V_s$ for some $s \in [q-1]_0$. If $i<j$, then $j \in V_{s-1}$, or $j=i+1$, or $i=0$ and $j=qp-3$. Analogously, if $j>i$, then $j \in V_{s+1}$, or $j=i-1$, or $i=qp-3$ and $j=0$.
\end{observation}

Thus, from Observation~\ref{obs:neighbors}, it follows that $ij$ for $i <j$ is an edge of $G_{q,p}$ if and only if 
\begin{itemize}
    \item $(i,j)$ is of type $(s,s-1)$, for some $s \in [q-1]_0$, or
    \item $i \in [qp-3]_0$ and $j=i+1$, or
    \item $i={0}$ and $j={qp-3}$.
\end{itemize}

We now show that graphs $G_{q,p}$ are minimal obstructions to $H$-coloring for a rich family of graphs $H$.

\begin{lemma}\label{lem:infinite-obstruction}
Let $q \geq 3$ be an odd integer, and let $H$ be graph of odd girth $q$ that does not contain $C_4$ as a subgraph.
For every $p \geq 1$ the graph $G_{q,p}$ is a minimal obstruction to $H$-coloring.
\end{lemma}
\begin{proof}
    To show that $G_{q,p}$ is a minimal obstruction to $H$-coloring, we need to show that
    (i) there is no homomorphism from $G_{q,p}$ to $H$,
    and that (ii) for each $v \in V(H)$ there is a homomorphism from $G_{q,p} - v$ to $H$.

    To see (i), assume that there exists a homomorphism $h$ from $G_{q,p}$ to $H$.
    \begin{claim}
        For each $j \in \{q, \ldots,qp-3\}$ we have $h(j)=h(j-q)$.
    \end{claim}
    \begin{claimproof}
   From the construction of $G_{q,p}$ it follows that each set of $q$ consecutive vertices $\{i,i+1,\ldots,i+q-1\}$ of $G_{q,p}$ induces an odd cycle of length $q$.
   This implies that their images $(h(i),h(i+1),\ldots,h(i+q-1))$ form a closed walk of length $q$ in $H$.
   As the odd girth of $H$ is $q$, we conclude that $\{h(i),h(i+1),\ldots,h(i+q-1)\}$ must induce an odd cycle of length $q$ in $H$.
   In particular, they must be mapped by $h$ to the distinct vertices of $H$.
   
    Assume the claim does not hold, and choose the smallest $j \geq q$ for which $h(j)\neq h(j-q)$.     
    Since $j-q,j-q+1$, and $j-1$ belong to a set of $q$ consecutive vertices of $G$, by the previous paragraph, $h(j-q),h(j-q+1)$, and $h(j-1)$ are three pairwise distinct vertices.
    But if now $h(j)$ and $h(j-q)$ are not mapped to the same vertex of $H$, vertices $h(j-q),h(j-q+1),h(j-1)$, and $h(j)$ form a (not-necessarily induced) $C_4$ in $H$, a contradiction. 
    \end{claimproof}
    
    The claim guarantees that for every $s \in[q-1]_0$, all the vertices that belong to $V_s$ are mapped to the same vertex of $H$, say $a_s$. By the fact that $h$ is a homomorphism, the set $\{a_0,a_1,\ldots,a_{q-1}\}$ induces a $q$-cycle in $H$.
    But now, since $0(qp-3) \in E(G_{q,p})$, we have that $h(0)h(qp-3)=a_0a_{q-3} \in E(H)$, a contradiction with the fact that the odd girth of $H$ is $q$.

    It remains to show (ii), i.e., for each $v \in V(G_{q,p})$ there is a homomorphism from $G_{q,p} - v$ to $H$.
    By symmetry of $G_{q,p}$, it is enough to consider the case when $v=qp-3$.
    As the odd girth of $H$ is $q$, there exists a $q$-cycle with consecutive vertices $a_0,\ldots,a_{q-1}$ in $H$.
    We claim that $h:V(G_{q,p}) \to V(H)$ defined as $h(i)=a_{i \mod q}$ for every $i \in V(G_{q,p})$ is a homomorphism.
    Indeed, take an edge $ij$ of $G$ and without loss of generality assume that $i<j$.
    Then by Observation~\ref{obs:neighbors}, we have either $h(i)h(j)=a_{i\bmod q}a_{i-1 \bmod q} \in E(H)$,
    or $h(i)h(j)=a_{i\bmod q}a_{i+1 \bmod q} \in E(H)$.
\end{proof}

\subsection{Excluded induced subgraphs of $G_{q,p}$}

Now let us show an auxiliary lemma that will be helpful in analyzing induced subgraphs that (do not) appear in $G_{q,p}$.
In particular, it implies that in order to prove that for each $p$, the graph $G_{q,p}$ is $F$-free for some graph $F$,
it is sufficient to show that this statement holds for some small values of $p$.

\begin{lemma}\label{lem:reducetofinite} 
Let $q$ be a fixed constant and $F$ be a graph.
Let $p \geq |V(F)|+2$.
If $G_{q,p-1}$ is $F$-free, then $G_{q,p}$ is $F$-free.
\end{lemma}
\begin{proof}
Assume otherwise, and let $U \subseteq V(G_{q,p})$ induce a copy of $F$ in $G_{q,p}$, in particular $|V(F)|=|U|$. 
Since $|V(F)| <qp-2$, we can assume without loss of generality that the vertex ${qp-3}$ does not belong to $U$.
Since $|V(G_{q,p})|=qp-2>q(|V(F)|+1)$, there exist $q+1$ consecutive vertices $\ell,\ldots,{\ell+q}$ that do not belong to $U$.
Define $R=\{j\in U~|~j>\ell+q\}$, $R'=\{{j-q}~|~j \in R\}$, and let $L=U\setminus R$.

Now consider $U'=L \cup R'$, and note that $\ell \notin U'$. It is straightforward to verify that $U' \subseteq V(G_{q,p-1})$.
We will show that $U'$ induces a copy of $F$ in $V(G_{q,p-1})$.
Since this is a contradiction with our assumption, we then conclude that $G_{q,p}$ is $F$-free.

Let $i,j \in U'$, let $s \in [q-1]_0$ be such that $i \in V_s$. Assume without loss of generality that $i <j$. 
Note that it is enough to show that
\begin{itemize}
    \item if $i,j < \ell$, then $ij \in E(G_{q,p-1})$ if and only if $ij \in E(G_{q,p})$, 
    \item if $i,j > \ell$, then $ij \in E(G_{q,p-1})$ if and only if $(i+q)(j+q) \in E(G_{q,p})$, 
    \item if $i < \ell < j$, then $ij \in E(G_{q,p-1})$ if and only if $i(j+q) \in E(G_{q,p})$.
\end{itemize}
The first item is straightforward.

For the second item, by Observation~\ref{obs:neighbors} we have that $ij \in E(G_{q,p-1})$ if and only if $j=i+1$ or $j\in V_{s-1}$. 
The first is equivalent to $j+q=(i+q)+1$, the latter is equivalent to ${j+q} \in V_{s-1}$.
Hence again using Observation~\ref{obs:neighbors} we obtain that  $ij \in E(G_{q,p-1})$ if and only if $({i+q})({j+q}) \in E(G_{q,p})$.

For the last item, note that $i < \ell < j$ implies $i \in L$ and $j \in R'$.
If $ij \in E(G_{q,p-1})$, then by Observation~\ref{obs:neighbors}, either $j=i+1$ or $v_j \in V_{s-1}$.
Since $i < \ell <j$, the first case is not possible.
In the second case, if $v_j \in V_{s-1}$, then ${j+q} \in V_{s-1}$, so $i({j+q}) \in E(G_{q,p})$. 
On the other hand, if $i({j+q}) \in E(G_{q,p})$, then, since $i < \ell <j$, it cannot happen that $j+q=i+1$.
If ${j+q} \in V_{s-1}$ then $j \in V_{s-1}$, so we conclude that $ij \in E(G_{q,p-1})$. 
That concludes the proof.
\end{proof}

The power of Lemma~\ref{lem:reducetofinite} is that in order to show that $G_{q,p}$ is $F$-free for \emph{every} $p$, it is sufficient to prove it for a finite (and small) set of graphs. This is encapsulated in the following, immediate corollary.

\begin{corollary}\label{cor:reducetofinite}
    Let $q$ be a fixed constant and $F$ be a graph. 
    If $G_{q,p}$ is $F$-free for every $p \leq |V(F)|+1$, then $G_{q,p}$ is $F$-free for every $p$.
\end{corollary}

Consequently, for every fixed $q$ and $F$, Corollary~\ref{cor:reducetofinite} reduces the problem of showing that $G_{q,p}$ is $F$-free to a constant-size task that can be tackled with a computer.

\subsection{Proof of Theorem~\ref{thm:infinite-gen} and Theorem~\ref{thm:infinite}}

Now let us analyze what induced paths and subdivided claws appear in $G_{q,p}$.
We start with showing that for every odd $q \geq 3$ and every $p \geq 1$ the graph $G_{q,p}$ is $qK_2$-free, i.e., they exclude an induced matching on $q$ edges. Here, an induced matching is a set of edges that are not only pairwise disjoint, but also non-adjacent.

\begin{lemma}\label{lem:qk2-free}
Let $q \geq 3$ be an odd integer.
For every $p \geq 1$ the graph $G_{q,p}$ is $qK_2$-free.
\end{lemma}
\begin{proof}
Let $M$ be a set of edges of $G_{q,p}$ that induces a matching. We aim to show that $|M| \leq q-1$.
\begin{claim}\label{cla:one-edge}
    For each $s \in [q-1]_0$, the set $M$ contains at most one edge of type $(s,s+1)$ or $(s+1,s)$.
\end{claim}
\begin{claimproof}
    Let $ij$ and ${i'}{j'}$ be some distinct edges in $M$ for $i<j$ and $i'<j'$.
    
    First, assume that $(i,j)$ and $(i',j')$ are of the same type $(a,b)$.
    By symmetry we can assume that $i<i'$ (since they belong to $M$, we cannot have $i = i'$).
    Observe that if $(a,b)=(s,s+1)$, then in particular $j=i+1$ and $j'=i'+1$ and thus $j<i'$.
    However, as $(j,i')$ is of type $(s+1,s)$, we conclude $ji'$ is an edge of $G_{q,p}$,
    so $M$ is not an induced matching, a contradiction.
    If, on the other hand, $(a,b)=(s+1,s)$, then $(i,j')$ is of type $(s+1,s)$. Thus $ij'$ is an edge of $G_{q,p}$, which again leads to a contradiction. 
    
    Now, let $ij, i'j'$ be the edges in $M$, such that $(i,j)$ is of type $(s,s+1)$ and $(i',j')$ is of type $(s+1,s)$.
    This means that $(j,j')$ is of the type $(s+1,s)$, that $(i,i')$ is of type $(s,s+1)$, and that $j=i+1$.
    Thus, either $i<j<i'<j'$ or $i'<i<j<j'$, or $i'<j'<i<j$.
    But then, in the first two cases, there is the edge $jj'$ in $G_{q,p}$, and in the last case we have the edge $i'i$,
    a contradiction with $M$ being an induced  matching.
\end{claimproof}

Note that there is at most one edge in $G_{q,p}$ of type $(0,q-3)$, namely $0(qp-3)$, thus there is at most one edge in $M$ of type $(0,q-3)$. If the edge $0(pq-3)$ belongs to $M$, then, since $M$ is an induced matching, no edge of the type $(0,q-1)$, $(q-1,0)$, $(0,1)$, nor $(1,0)$ belongs to $M$. Thus, by Claim~\ref{cla:one-edge}, $|M|\leq q-1$.

Therefore, we can assume that the edge $0(pq-3)$ does not belong to $M$. 
In this case, for every $s\in[q-1]_0$, the only possible types of an edge in $G_{q,p}$ that have $s$ on the first coordinate are $(s,s+1)$ and $(s,s-1)$, and for $t \in \{s+1,s-1\}$ there is at most one edge of type $(s,t)$ or $(t,s)$.
This in particular means that $|M| \leq q$.

\begin{claim}\label{clm:ordering}
Let $s \in [q-1]_0$, and let $ij, i'j'$ be the edges in $M$ such that $i <j$ and $i'<j'$.
    If $(i,j)$ is of type $(s,s+1)$ or $(s+1,s)$, and $(i',j')$ is of type $(s+1,s+2)$ or $(s+2,s+1)$, then $i <i'$.
\end{claim}
\begin{claimproof}
    Assume otherwise, i.e., that $i'<i$. If $(i',j')$ is of type $(s+1,s+2)$, then we must have $j'=i'+1$, thus $i'<j'<i<j$.
    Therefore there is an edge $i'i$ (if $(i,j)$ is of type $(s+1,s+2)$) or an edge $i'j$ (if $(i,j)$ is of type $(s+2,s+1)$), in both cases a contradiction.

    Now consider the case when $(i',j')$ is of type $(s+2,s+1)$. In that case, however, either there is an edge $i'i$ (if $(i,j)$ is of type $(s+1,s+2)$) or an edge $i'j$ (if $(i,j)$ is of type $(s+2,s+1)$), again a contradiction.
\end{claimproof}

Suppose that $|M|=q$. This means that for every $s \in [q-1]_0$ there is an edge $ij$ in $M$ such that $(i,j)$ is of type $(s,s+1)$ or $(s+1,s)$.
Let $ij \in M$, where $i < j$, be such that $i$ is smallest possible (i.e., $ij$ is the ``first'' edge of $M$).
Let the type of $(i,j)$ be $(s,s+1)$ or $(s+1,s)$.
Let $i'j'$, where $i' < j'$, be the edge of $M$ such that $(i',j')$ is of type $(s-1,s)$ or $(s,s-1)$; it exists by our assumption.
By Claim~\ref{clm:ordering}, we conclude that $i' < i$, which contradicts the choice of $ij$.
Thus $|M|<q$, which concludes the proof.
\end{proof}

Let us remark that Lemma~\ref{lem:qk2-free} is best possible, i.e., if $p$ is large enough, then $G_{q,p}$ contains $(q-1)K_2$ as an induced subgraph.
We do not prove it, as later, in Lemma~\ref{lem:lowerpath}, we will show a stronger result. 
Let us turn our attention to induced paths and subdivided claws that do not appear in $G_{q,p}$.

\begin{lemma}\label{lem:infinite-easybounds}   
    Let $q \geq 3$ be an odd integer.    
    For every $p \geq 1$ the graph $G_{q,p}$ is $\{P_{3q-1}, S_{2,2,2},$ $S_{\nicefrac{3(q-1)}{2},\nicefrac{3(q-1)}{2},1} \}$-free.
\end{lemma}
\begin{proof}
Notice that an induced path with consecutive vertices $p_0,p_1,\ldots,p_{3q-2}$ contains $qK_2$ as an induced subgraph (with vertex set $\{p_ip_{i+1}~|~i=0 \mod 3\}$).
Similarly, $S_{\nicefrac{3(q-1)}{2},\nicefrac{3(q-1)}{2},1}$ contains $qK_2$ is an induced subgraph; see Figure~\ref{fig:corollary}.
Thus the fact that $G_{q,p}$ is $\{P_{3q-1}, S_{\nicefrac{3(q-1)}{2},\nicefrac{3(q-1)}{2},1} \}$-free follows directly from  Lemma~\ref{lem:qk2-free}.

\begin{figure}
    \centering
    \includegraphics[page=5,scale=1]{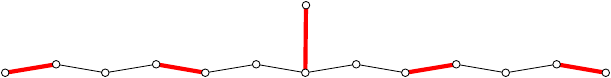}
    \caption{If a graph contains $S_{\nicefrac{3(q-1)}{2},\nicefrac{3(q-1)}{2},1}$ as an induced subgraph, then it also contains $qK_2$.}
    \label{fig:corollary}
\end{figure}

Now let us argue that $G_{q,p}$ is also $S_{2,2,2}$-free.
The case $q=3$ requires slightly different arguments that he case $q \geq 5$.
However, thanks to Corollary~\ref{cor:reducetofinite}, in order to show that for every $p$, the graph $G_{3,p}$ is $S_{2,2,2}$-free, it is sufficient to verify this for $p \leq 8$. We confirmed this using exhaustive computer search.

So from now on assume that $q \geq 5$.
For contradiction suppose that $G_{q,p}$ has an induced subgraph $T$ isomorphic to $S_{2,2,2}$. 
By symmetry of $G_{q,p}$ we can assume that 2 is the unique vertex of $T$ of degree 3, let $i<j<k$ be the neighbors of 2 in $T$,
and let $i',j'$ and $k'$ be their respective neighbors of degree 1 in $T$.
By the definition of $G_{q,p}$, either $i=1$, or $i,j,k >2$.

Consider the case $i=1$. If $(2,j)$ and $(2,k)$ are both of type $(2,1)$, then either $(1,j')$ is of type $(1,0)$ or $(1,2)$.
In the first case there is an edge $ij'$, in the second case $j'\leq j+1 < k$ and thus there is an edge $j'k$, thus we reach a contradiction.
Thus $(2,j)$ is of type $(2,3)$ and $(2,k)$ is of type $(2,1)$. But in that case $k'>j$, and either $(1,k')$ is of type $(1,0)$ and $ik'$ is an edge of $G_{q,p}$, or $(3,k')$ is of type $(3,2)$ and $jk'$ is an edge of $G_{q,p}$, again a contradiction.

Thus we can assume that $i,j,k >2$, in particular $(2,j)$ and $(2,k)$ are of type $(2,1)$.
If $(2,i)$ is also of type $(2,1)$, then $(j',j)$ is either of type $(2,1)$, or of type $(0,1)$.
In the first case, $j'\leq j+1<k$ and there is an edge $j'k$ in $G_{q,p}$.
In the second case, $j'\geq j-1>i$ and as $(i,j')$ is of type $(1,0)$, thus $i'j$ is an edge of $G_{q,p}$, in both cases we reach a contradiction.

It remains to consider the case when $(2,i)$ is of type $(2,3)$, i.e., $i'=i+1=3 < j,j',k,k'$.
Then, $(k,k')$ is either of type $(1,2)$, or of type $(1,0)$. In the first case, $(i',k')$ is of type $(3,2)$, and thus $i'k'$ is an edge of $G_{q,p}$.
In the second case we have $k'\geq k-1 >j$ and $(j,k')$ is of type $(1,0)$, thus $jk'$ is an edge of $G_{q,p}$, a contradiction.
\end{proof}
Now, as an immediate consequence of Lemma~\ref{lem:infinite-obstruction} and Lemma~\ref{lem:infinite-easybounds}, we obtain Theorem~\ref{thm:infinite-gen}, which we restate below.

\thminfinitepathsgen*

Note that for $H = C_5$, i.e., for $q=5$, Theorem~\ref{thm:infinite-gen} shows that the constructed graphs are in particular $P_{14}$-free and $S_{6,6,1}$-free.
It turns out that they are actually $P_{13}$-free and $S_{5,5,1}$-free (and also exclude some other subdivided claws). Here we will make use of Corollary~\ref{cor:reducetofinite}, combined with computer search.

Let us start with analyzing the length of a longest induced path in $G_{5,p}$.
Thus we are interested in applying  Corollary~\ref{cor:reducetofinite} to the case $q=5$ and $F=P_{13}$. Actually,  we can even exclude $F=P_{10}+P_2$, i.e., the graph with two components: one isomorphic to $P_{10}$ and the other isomorphic to $P_2$.
Note that $(P_{10}+P_2)$-free graphs are in particular $P_{13}$-free.
Furthermore, the graph $G_{5,p}$ excludes the following subdivided claws: $S_{5,5,1}$, $S_{11,1,1}$, and $S_{8,2,1}$.
These results, together with Lemma~\ref{lem:infinite-obstruction}, give us Theorem~\ref{thm:infinite}.

\thminfinite*

\subsection{Longest induced paths in $G_{q,p}$}

Theorem~\ref{thm:infinite}, and the fact that from the result of Chudnovsky et al.~\cite{DBLP:journals/jct/ChudnovskyGSZ20}
it follows that for every $p$, the graph $G_{3,p}$ is $P_7$-free (while Lemma~\ref{lem:infinite-easybounds} only gives $P_8$-freeness),
suggest that the bound on the length of a longest induced path given by Theorem~\ref{thm:infinite-gen} is not optimal also in the other cases. This evidence suggests the following conjecture.

\begin{conjecture}\label{con:path}
    Let $q \geq 3$ be odd. For every $p \geq 1$, the graph $G_{q,p}$ is $P_{3q-2}$-free. 
\end{conjecture}

We conclude this section by showing that the bound from Conjecture~\ref{con:path}, if true, is best possible.

\begin{lemma}\label{lem:lowerpath}
Let $q \geq 3$ be odd. If $p \geq 2q+1$, then $G_{q,p}$ contains $P_{3q-3}$ as an induced subgraph.
\end{lemma}
\begin{proof}
Let $q$ and $p$ be as in the statement of the lemma.
We will define a sequence of $3q-3$ vertices \[P=(a_0,a_1,a_2,a_3,\ldots,a_{3q-4})\] that induce a path in $G_{q,p}$.
For $j \in [2]_0$ let $Q_j=\{i~|~i = j \mod 3, i \in [3q-4]\}$. 
Let $a_0=0$, and for $i \in [3q-4]$ we set
\begin{align*}
    a_i=\begin{cases}
        q^2-1 - k(q+1) & \textrm{ if }i \in Q_1, k=(i-1)/3, \\
        2q^2-2-k(q+1) & \textrm{ if }i \in Q_2, k=(i-2)/3. \\
        2q^2-3-k(q+1) & \textrm{ if }i \in Q_0, k=(i-3)/3, \\
    \end{cases}
\end{align*}
Observe that since $p\geq 2q+1$, the graph $G_{q,p}$ has at least $2q^2+q-2$ vertices, so the vertices $a_i$ are well-defined. Observe that for every $i \in [3q-5]$, if $i = 1 \mod 3$ then $a_i < a_{i+1}$, and otherwise $a_i > a_{i+1}$. In particular, if $i = 2 \mod 3$, then $a_{i+1}=a_i-1$.
Moreover, from the definition it follows that the sequence of the types of the consecutive vertices of $P$ is \[(0,q-1,q-2,q-3,q-2,q-3,q-4,q-3,q-4,q-5,\ldots,2,1,0,1,0).\]
Now to see that $P$ is an induced path, we show that for every $i \in [3q-5]_0$ that is of type $s \in [q-1]_0$ the edge $a_ia_{i+1}$ exists and no edge of the form $a_ia_j$ for $j>i+1$ exists.

\smallskip
\noindent\textbf{\boldmath Case 1: $i=0$.} Observe that $0a_1\in E(G_{q,p})$ since $(0,a_1)$ is of type $(0,q-1)$. On the other hand, for every $i \in \{2,\ldots, 3q-4\}$ we have $0a_i \notin E(G_{q,p})$, as $(0,a_i)$ is not of type $(0,q-1)$. 

\smallskip
\noindent\textbf{\boldmath Case 2: $i \in Q_1$.} Observe that for every $j \in [3q-4] \setminus \{i\}$, if $j \in Q_0 \cup Q_2$ or $j \in Q_1, j<i$, then $a_i < a_j$. Otherwise (if $j \in Q_1$ and $j>i$), $a_i>a_j$. Thus, as $(a_i,a_{i+1})$ is of type $(s,s-1)$, the edge $a_ia_{i+1}$ belongs to $E(G_{q,p})$. On the other hand, there is no $j >i+1$ such that (i) $a_j>a_i$ and $(a_i,a_j)$ is of type $(s,s-1)$, or (ii) $a_j<a_i$ and $(a_i,a_j)$ is of type $(s,s+1)$, thus $a_ia_j \notin E(G_{q,p})$ for any $j >i+1$.

\smallskip
\noindent\textbf{\boldmath Case 3: $i \in Q_2$.} Observe that for every $j \in [3q-4] \setminus \{i\}$, if $j \in Q_1$ or $j>i$, then $a_i > a_j$. Otherwise, $a_i>a_j$. As $a_{i+1}=a_i-1$, the edge $a_ia_{i+1}$ belongs to $E(G_{q,p})$. On the other hand, there is no $j >i+1$ such that (i) $a_j>a_i$ or (ii) $a_j<a_i$ and $(a_i,a_j)$ is of type $(s,s+1)$. Thus $a_ia_j \notin E(G_{q,p})$ for any $j >i+1$.

\smallskip
\noindent\textbf{\boldmath Case 3: $i \in Q_0$.}  Observe that for every $j \in [3q-4] \setminus \{i\}$, if $j \in Q_0 \cup Q_2$ and $j <i$, then $a_i < a_j$. Otherwise (if $j \in Q_1$ or $j \in Q_0 \cup Q_2$ and $j>i$), $a_i>a_j$. Thus, as $(a_i,a_{i+1})$ is of type $(s,s+1)$, the edge $a_ia_{i+1}$ belongs to $E(G_{q,p})$. On the other hand, there is no $j >i+1$ such that (i) $a_j>a_i$ or (ii) $a_j<a_i$ and $(a_i,a_j)$ is of type $(s,s+1)$. Thus $a_ia_j \notin E(G_{q,p})$ for any $j >i+1$.
\end{proof}

\section{Conclusion} 
We conclude the paper with pointing out some possible directions for further research.

A very natural question is to close the gaps concerning the finiteness of the family of minimal obstructions to $C_5$-coloring among $F$-free graphs, when $F$ is a path or a subdivided claw.
Note that in the former case, the open cases are $H = P_t$ where $t \in [9,12]$, and in the latter case, $H = S_{a,b,1}$ where $a \geq b \geq 1$ and (i) $b=1$ and $4 \leq a \leq 10$, or (ii) $b=2$ and $3 \leq a \leq 7$ or (iii) $b \in \{3,4\}$ and $b \leq a \leq 11-b$.

If it comes to $H$-coloring for $H$ other than a clique or $C_5$, then the situation is much less clear. Indeed, recall that in general, we do not know a complete classification of the complexity of $H$-coloring in $F$-free graphs. In spite of that, there are still several research directions that may be worth following.

First, we believe it would be interesting to classify what happens in $P_t$-free graphs for small fixed values of $t$, with a particular focus on the smallest non-trivial case, namely the class of $P_5$-free graphs. Note that we already know that if $H$ is an odd cycle, then the family of $P_5$-free minimal obstructions to $H$-coloring is finite~\cite{DBLP:journals/dam/KaminskiP19}, and it is infinite if $H$ is a clique on at least four vertices~\cite{DBLP:journals/dam/HoangMRSV15}. 
The foregoing study on $H$-coloring $P_t$-free graphs suggests that if $H$ is \emph{prime} (i.e., admits no non-trivial decomposition with respect to the so-called direct product~\cite{hammack2011handbook}), then the problem behaves quite differently depending whether $H$ contains a $C_4$ as a subgraph or not.
In order to stimulate research in this direction, we suggest the following (bold) conjecture.

\begin{conjecture}\label{con:c4-infinite}
If $H$ is a prime, non-bipartite core, there is a finite family of minimal obstructions to $H$-coloring $P_5$-free graphs if and only if $H$ does not contain $C_4$ as a subgraph.
\end{conjecture}

\paragraph{Acknowledgements.}
The authors are grateful to R\'emi Di Guardia for fruitful and inspiring discussions that led to this project.
This project itself was initiated at Dagstuhl Seminar 22481 ``Vertex Partitioning in Graphs: From Structure to Algorithms.''
We are grateful to the organizers and other participants for a productive atmosphere.
Jorik Jooken is supported by an FWO grant with grant number 1222524N. Jan Goedgebeur is supported by Internal Funds of KU Leuven and an FWO grant with grant number G0AGX24N. 
Karolina Okrasa is supported by a Polish National Science Centre grant no. 2021/41/N/ST6/01507.

\bibliographystyle{abbrv}
\bibliography{refs}

\appendix
\section*{Appendix: Implementation and testing of Algorithm~\ref{algo:expand}}\label{sec: appendix}

\paragraph{Details related to the efficient implementation of Algorithm~\ref{algo:expand}.}\label{sec:app1}
Some care is needed in order to make the implementation of Algorithm~\ref{algo:expand} run fast in practice. For example, if it is necessary to verify whether or not a graph simultaneously has properties A and B, it is not necessary to verify property B if one already knows that property A does not hold. Moreover, the time needed to verify one property can be drastically different from the time needed to verify another one. For this reason, we experimentally determined that it is beneficial to first test for $F$-freeness, then test for $H$-colorability (if still necessary), and finally test for isomorphism (if still necessary). This agrees with the experimental observations from~\cite{DBLP:journals/jgt/GoedgebeurS18} for $k$-coloring. For isomorphism testing, we used the program \texttt{nauty}~\cite{DBLP:journals/jsc/McKayP14} for computing the canonical form of a graph. This is stored in a balanced binary search tree to allow efficient insertion and lookup.

Since the graphs that we deal with in this paper are all small, we represent sets internally as bitvectors. This allows the hardware to efficiently support standard operations such as the union or intersection of sets. For example, when the algorithm adds a new vertex $x$ and edges between $x$ and previous vertices in all possible ways, the algorithm keeps track of a set of edges that would result in an induced $K_3$ after adding the edge. Bitvectors allow the algorithm to efficiently iterate over only those edges that would not result in an induced $K_3$ (note that, except for $K_3$, all minimal obstructions to $C_5$-coloring are $K_3$-free). The source code of our algorithm is made publicly available at~\cite{gitrepo}.

\paragraph{Independent verifications and sanity checks.}
Apart from Algorithm~\ref{algo:expand}, we also implemented other algorithms to independently verify some of our results and reproduce existing results from the literature. All results are in agreement and this provides strong evidence for the correctness of our implementation. More precisely, we adapted our algorithm to find minimal obstructions to $H$-coloring which are $\{F_1,F_2\}$-free for several triples $(H,F_1,F_2)$ of graphs.
Amongst others, this allowed us to reproduce already known exhaustive lists of
$\{P_5,\text{dart}\}$-free minimal obstructions to $k$-coloring (for $k \in \{4,5,6\}$)~\cite{DBLP:conf/cocoa/XiaJGH23}, 
$\{P_5,\text{gem}\}$-free minimal obstructions to $k$-coloring (for $k \in \{3,4\}$)~\cite{DBLP:journals/dam/CaiGH23}, 
$\{P_5,C_5\}$-free minimal obstructions to $4$-coloring~\cite{DBLP:journals/dam/HoangMRSV15} and $\{P_5,H\}$-free minimal obstructions to $4$-coloring for each $H \in \{K_{1,3}+P_1,K_{1,4}+P_1,\overline{K_3+2P_1}\}$~\cite{DBLP:journals/corr/abs-2403-05611}. In each of these cases, our algorithm terminates. Apart from exhaustive generation, we also used our algorithm to reproduce all $P_6$-free minimal obstructions to $3$-coloring until order 13~\cite{DBLP:journals/jct/ChudnovskyGSZ20} and all $\{P_7,C_3\}$-free minimal obstructions to $3$-coloring until order 16~\cite{DBLP:journals/jgt/GoedgebeurS18}. We also adapted the algorithm \texttt{geng}~\cite{DBLP:journals/jsc/McKayP14} to generate the same minimal obstructions for $C_5$-coloring as our Algorithm~\ref{algo:expand}. However, note that by only using \texttt{geng}, one cannot conclude that there are no larger minimal obstructions than those found by \texttt{geng}. This further motivates the need for a specialised algorithm (cf.\ Algorithm~\ref{algo:expand}).

We also implemented two independent algorithms for verifying whether a graph $G$ can be $C_5$-colored.
The first one is a simple recursive backtracking algorithm that colors the vertices of $G$ one by one, but backtracks as soon as no valid colors can be found for a given vertex.
The second one formulates the problem as an integer program and uses a mixed integer programming solver to verify if there are any feasible solutions. A graph $G$ can be $C_5$-colored if and only if the following constraints can be satisfied:
\begin{align}
\sum_{v \in V(C_5)}x_{u,v} &= 1 &\forall u \in V(G) \label{constrA} \\
x_{u,v} &\leq \sum_{v' \in N_{C_5}(v)} x_{u',v'} &\forall u \in V(G), u' \in N_G(u) \label{constrB}\\
x_{u,v} &\in \{0,1\} &\forall u \in V(G), v \in V(H) \label{constrC}
\end{align}
Here, constraints \eqref{constrC} indicate the binary domains of the variables $x_{u,v}$; such a variable will be equal to 1 if vertex $u \in V(G)$ is mapped to vertex $v \in V(C_5)$.
Constraints~\eqref{constrA} ensure that every vertex $u \in V(G)$ must be mapped to precisely one vertex $v \in V(C_5)$.
Finally, constraints~\eqref{constrB} ensure that if a vertex $u \in V(G)$ is mapped to vertex $v \in V(C_5)$,
then all neighbors $u' \in V(G)$ must be mapped to neighbors $v'$ of $v$ in $C_5$.
\end{document}